\documentclass[11pt]{amsart}
\usepackage[english]{babel} 
\usepackage[latin1]{inputenc} 
\usepackage{amssymb,geometry,color,graphicx}
\usepackage{amsmath,amsthm} 
\usepackage{booktabs}
\usepackage{datetime}
\usepackage{dsfont}  
\usepackage{amstext} 
\usepackage{graphicx}   
\usepackage{fancyhdr} 
\usepackage{latexsym}
\usepackage{mathrsfs}
\usepackage{cases}
\usepackage{mathtools}
\usepackage{bm}
\usepackage{caption}
\usepackage{subcaption}
\usepackage{tikz}
\usepackage{hyperref}
\usepackage{enumitem}
\hypersetup{
    colorlinks=true, %set true if you want colored links
    linktoc=all,     %set to all if you want both sections and subsections linked
    linkcolor=blue,  %choose some color if you want links to stand out
}

\usepackage{changes}
\definechangesauthor[name=Pawel, color=blue]{PP}
\definechangesauthor[name=Fabio, color=teal]{FD}

\numberwithin{equation}{section} 

\newcommand{\udef}{\mathrel{\mathop:}=}

\newcommand{\R}{\mathbb{R}}

\newcommand{\N}{\mathbb{N}}

\newcommand{\de}{\mathrm{d}}

\theoremstyle{plain}
\newtheorem{thm}{Theorem}[section]

\newtheorem{prop}[thm]{Proposition}
\newtheorem{lem}[thm]{Lemma}

\newtheorem{rem}[thm]{Remark}
\newtheorem{exm}[thm]{Example}

    \title[Existence, uniqueness and approximation of solutions to CDDEs]{Existence, uniqueness and approximation of solutions to Carath\'eodory delay differential equations}
\author[F.V. Difonzo]{Fabio~V. Difonzo}
\address{Dipartimento di Matematica, Universit\`a degli Studi di Bari Aldo Moro, Via E. Orabona 4, 70125 Bari, Italy}
\email{fabio.difonzo@uniba.it}

\author[P. Przyby{\l}owicz]{Pawe{\l} Przyby{\l}owicz}
\address{AGH University of Science and Technology,
Faculty of Applied Mathematics,
Al. A.~Mickiewicza 30, 30-059 Krak\'ow, Poland}
\email{pprzybyl@agh.edu.pl, corresponding author}

\author[Y. Wu]{Yue Wu}
\address{Department of Mathematics and Statistics, University of Strathclyde, Glasgow, UK}
\email{yue.wu@strath.ac.uk}

\begin{document}
\begin{abstract}
    In this paper we address the existence, uniqueness and approximation of solutions of delay differential equations (DDEs) with Carath\'eodory type right-hand side functions. We provide construction of randomized Euler scheme for DDEs and investigate its error. We also report results of numerical experiments.
\end{abstract}
\maketitle

\section{Introduction}
We deal with approximation  of solutions to the following delay differential equations (DDEs)
\begin{equation}
\label{eq:DiscDDE2}
\begin{cases}
x'(t)=f(t, x(t),x(t-\tau)), & t\in [0,(n+1)\tau], \\
x(t)=x_{0}, & t\in[-\tau,0),
\end{cases}
\end{equation}
with the constant time-lag $\tau\in(0,+\infty)$, fixed time horizon $n\in\mathbb{N}$,  $f:[0,(n+1)\tau]\times\R^d\times\R^d\mapsto\R^d$, and $x_0\in\R^d$. We assume that $f$ is integrable with respect to $t$ and (at least) continuous with respect to $(x,z)$. Hence, we consider Carath\'eodory type conditions for $f$.

Motivation of considering such DDEs comes, for example,  from the problem of modeling switching systems with memory, see \cite{hale1977theory}, \cite{Hale1993IntroductionTF}. Moreover, another inspiration follows from delayed differential equations with rough paths of the form
\begin{equation}\label{eq:DiscDDE21}
\begin{cases}
\de U(t)=a(U(t),U(t-\tau))+\de Z(t), & t\in [0,(n+1)\tau], \\
U(t)=U_{0}, & t\in[-\tau,0),
\end{cases}
\end{equation}
where $Z$ is an integrable perturbation (which might be of stochastic nature). Then $x(t)=U(t)-Z(t)$ satisfies the (possibly random) DDE \eqref{eq:DiscDDE2} with $f(t,x,z)=a(x+Z(t),z+Z(t-\tau))$ and $x(t)=U_0$, where we assume that $Z(t)=0$ for $t\in [-\tau,0]$. In this case the function $f$ inherits from $Z$ its low smoothness with respect to the variable $t$. (The exemplary equation \eqref{eq:DiscDDE21} is a generalization of the ODE with rough paths discussed in \cite{RKYW2017}.)

In the case of classical assumptions (such as $C^r$-regularity of $f=f(t,x,z)$ wrt all variables $t,x,z$) imposed on the right-hand side function errors for deterministic schemes have been established, for example, in the book \cite{bellen1}, which is the standard reference. See also \cite{CZPMPP}, where  error of the Euler scheme has been investigated for some class of nonlinear DDEs under nonstandard assumptions, such as one-side Lipschitz condition and local H\"older continuity. In contrast, much less is known about the approximation of solutions of  DDEs with less regular Carath\'eodory right-hand side function $f$. In the case of Carath\'eodory ODEs we need to turn to randomized algorithms, such as randomized Euler scheme, since it is well-known that there is lack of convergence for deterministic algorithms, see \cite{RKYW2017}. This behavior is inherited by DDEs, since ODEs form a subclass of DDEs. Hence, we define randomized version of the Euler scheme that is suitable for DDEs of the form \eqref{eq:DiscDDE2}.

While the randomized algorithms for ODEs have been widely investigated in the literature (see, for example, \cite{bochacik2}, \cite{BGMP2021}, \cite{daun1}, \cite{hein_milla1}, \cite{jen_neuen1}, \cite{BK2006}, \cite{RKYW2017}), according to our best knowledge this is the first paper that defines randomized Euler scheme and rigorously investigates its error for (Carath\'eodory type)  DDEs.

The main contributions of the paper are as follows:
\begin{itemize}
    \item we show existence, uniqueness, and H\"older regularity of the solution to \eqref{eq:DiscDDE2} when the right-hand side function $f=f(t,x,z)$ is only integrable with respect to $t$ and satisfies local Lipschitz assumption with respect to $(x,z)$ (Theorem \ref{sol_dde_prop_1}), 
    \item we perform rigorous error analysis of the randomized Euler scheme applied to \eqref{eq:DiscDDE2} when the right-hand side function $f=f(t,x,z)$ is only integrable with respect to $t$, satisfies global Lipschitz condition with respect to $x$ and it is globally H\"older continuous with respect to $z$ (Theorem \ref{rate_of_conv_expl_Eul}),
    \item we report results of numerical experiments that show stable error behaviour as stated in Theorem \ref{rate_of_conv_expl_Eul}.
\end{itemize}
In addition, as a consequence of Theorem \ref{sol_dde_prop_1} we establish almost sure convergence of the randomized Euler scheme, see Proposition \ref{as_conv}.

We want to stress here that the techniques used when proving upper error bounds in  Theorem  \ref{rate_of_conv_expl_Eul} differ significantly comparing to that used in \cite{bochacik2}, \cite{BGMP2021}, \cite{daun1}, \cite{hein_milla1}, \cite{jen_neuen1}, \cite{RKYW2017} for randomized  algorithms defined for ordinary differential equations. Mainly, due to the fact that DDEs have to be considered interval-by-interval we developed a suitable proof technique that is based on mathematical induction. In particular, suitable inductive assumptions have to be related with the H\"older continuity of $f=f(t,x,z)$ with respect to the (delayed) variable $z$. It can be observed from the proof that the main result from \cite{RKYW2017} can be applied only for the initial inductive step, when the H\"older continuity of $f$ with respect to $z$ does not make an impact, equivalently, when the CDDE and its numerical counter part are subject to the same value of the delay variable. As the systems iterate over time, a different strategy is required to handle the effect of the delay variable. An auxiliary randomized Euler scheme, which is not practical but theoretically useful, is therefore introduced for error decomposition to leverage the inductive estimate from previous steps as well as the randomized quadrature rule from \cite{RKYW2017}. 

The structure of the article is as follows. Basic notions, definitions together with assumptions and definition of the randomized Euler scheme are given in Section 2. All Section 3 is devoted to the issue of existence and uniqueness of solutions of the  Carath\'eodory type DDEs \eqref{eq:DiscDDE2} in the case when $f=f(t,x,z)$ is only integrable with respect to $t$ and satisfies local Lipschitz assumption with respect to $(x,z)$. Section 4 contain proof of the main result of the paper (Theorem \ref{sol_dde_prop_1}) that states upper bounds on the error of the randomized Euler scheme. In Section 5 we report results of numerical experiments. Finally, Appendix contains auiliary results for Carath\'eodory type ODEs that we use in the paper.

%\section{Carath\'eodory DDEs and randomized Euler algorithm}

\section{Preliminaries}
 By $\|\cdot\|$ we mean the Euclidean norm in $\R^d$.
We consider a  complete probability space $(\Omega,\Sigma,\mathbb{P})$. For a random variable $X:\Omega\to\mathbb{R}$ we denote by $\|X\|_{L^p(\Omega)}=(\mathbb{E}|X|^p)^{1/p}$, where $p\in [2,+\infty)$.

Let us fix the {\it horizon parameter} $n\in\mathbb{N}$. On the right-hand side function $f$ we impose the following assumptions:
\begin{enumerate}[label=\textbf{(A\arabic*)},ref=(A\arabic*)]
    \item\label{ass:A1} $f(t,\cdot,\cdot)\in C(\R^d\times\R^d;\R^d)$ for all $t\in [0,(n+1)\tau]$,
    \item\label{ass:A2} $f(\cdot,x,z)$ is Borel measurable for all $(x,z)\in \R^d\times\R^d$,
    \item\label{ass:A3}
    there exists $K:[0,(n+1)\tau]\to [0,+\infty)$ such that $K\in L^1([0,(n+1)\tau])$ and for all $(t,x,z)\in [0,(n+1)\tau]\times\R^d\times\R^d$
    \begin{equation}
     \label{assumpt_a3}
        \|f(t,x,z)\|\leq K(t) (1+\|x\|)(1+\|z\|),
    \end{equation}
    \item\label{ass:A4}
    for every compact set $U\subset\R^d$ there exists $L_U:[0,(n+1)\tau]\mapsto [0,+\infty)$ such that $L_U\in L^1([0,(n+1)\tau])$ and for all $t\in [0,(n+1)\tau]$, $x_1,x_2\in U$, $z\in\R^d$
    \begin{equation}
    \label{assumpt_a4}
        \|f(t,x_1,z)-f(t,x_2,z)\|\leq L_U(t) (1+\|z\|)\|x_1-x_2\|.
    \end{equation}
\end{enumerate}
In Section 3,  under the assumptions \ref{ass:A1}-\ref{ass:A4}, we investigate existence and uniqueness of solution for \eqref{eq:DiscDDE2}. Next, in Section 4 we investigate error of the {\it randomized Euler scheme} under slightly stronger assumptions. Namely, we impose global Lipschitz assumption on $f=(t,x,z)$ with respect to $(x,z)$  instead of its local version \ref{ass:A4}.

The mentioned above  randomized Euler scheme is defined as follows. Fix the {\it discretization parameter} $N\in\N$ and set
\begin{displaymath}
	t_k^j=j\tau+kh, \quad k=0,1,\ldots,N, \ j=0,1,\ldots,n,
\end{displaymath} 
 where
\begin{equation}
	h=\frac{\tau}{N}.
\end{equation}
Note that for each $j$ the sequence $\{t^j_k\}_{k=0}^N$ provides uniform discretization of the subinterval $[j\tau,(j+1)\tau]$.
Let $\{\gamma_k^j\}_{j\in\mathbb{N}_0,k\in\mathbb{N}}$ be an iid sequence of random variables, defined on the complete probability space $(\Omega,\Sigma,\mathbb{P})$, where every $\gamma_k^j$ is uniformly distributed on $[0,1]$. We set $y_0^{-1}=\ldots=y^{-1}_N=x_0$ and then  for $j=0,1,\ldots,n$, $k=0,1,\ldots,N-1$ we take
\begin{eqnarray}
\label{expl_euler_1}
	&&y_0^j=y^{j-1}_N,\\
\label{expl_euler_11}	
	&&y_{k+1}^j= y_k^j+h\cdot f(\theta_{k+1}^j,y_k^j,y_k^{j-1}),
\end{eqnarray}
where $\theta_{k+1}^j=t_k^j+h\gamma_{k+1}^j$. As the output we obtain the  sequence of $\mathbb{R}^d$-valued random vectors  $\{y_k^j\}_{k=0,1,\ldots,N}, j=0,1,\ldots,n$ that provides a discrete approximation of the values $\{x(t_k^{j})\}_{k=0,1,\ldots,N}, j=0,1,\ldots,n$. It is easy to see that each random vector $y_k^j$, $j=0,1,\ldots,n$, $k=0,1,\ldots,N$, is measurable with respect to the $\sigma$-field generated by the following family of independent random variables
\begin{equation}
\label{sig_filed1}
\Bigl\{\theta_1^0\ldots,\theta_N^0,\ldots,\theta_1^{j-1},\ldots,\theta_N^{j-1},\theta_1^j,\ldots,\theta_k^j\Bigr\}.
\end{equation}

As the horizon parameter $n$ is fixed, the randomized Euler scheme uses $O(N)$ evaluations of $f$ (with a constant
in the '$O$' notation that only depends on $n$ but not on $N$).

In Section 4 we provide upper bounds on the error
\begin{equation}\label{eq:L2err}
    \Bigl\|\max\limits_{0\leq i\leq N}\|x(t_i^j)-y_i^j\|\Bigl\|_{L^p(\Omega)}
\end{equation}
for $j=0,1,\ldots,n$.
%%%%%%%%%%%%%%%%%%
\section{Properties of solutions to  Carath\'eodory DDEs}
In this section we investigate the issue of existence and uniqueness of the solution of \eqref{eq:DiscDDE2} under the assumptions \ref{ass:A1}-\ref{ass:A4}. 

In the sequel we use the following equivalent representation of the solution of \eqref{eq:DiscDDE2}, that is very convenient when proving its properties and when estimating the error of the randomized Euler scheme. For $j=0,1,\ldots,n$ and $t\in [0,\tau]$ it holds
\begin{equation}
    x'(t+j\tau)=f(t+j\tau,x(t+j\tau),x(t+(j-1)\tau)).
\end{equation}
Hence, we take $\phi_{-1}(t):=x_0$,  $\phi_j(t)\udef x(t+j\tau)$ and for $j=0,1,\ldots,n$ we consider the following sequence of initial-value problems
\begin{equation}
\label{eqODE_j0}
\begin{cases}
\phi_j'(t)=g_j(t,\phi_j(t)), & t\in [0,\tau], \\
\phi_j(0)=\phi_{j-1}(\tau), 
\end{cases}
\end{equation}
with $g_j(t,x)=f(t+j\tau,x,\phi_{j-1}(t))$, $(t,x)\in [0,\tau]\times\mathbb{R}^d$. Then the solution of \eqref{eq:DiscDDE2} can be written as
  \begin{equation}
        x(t)=\sum\limits_{j=-1}^n \phi_j(t-j\tau)\cdot\mathbf{1}_{[j\tau,(j+1)\tau]}(t), \quad t\in [-\tau,(n+1)\tau].
    \end{equation}
%%%%%%%%%%%%%%%%%%

We prove the following result about existence, uniqueness and H\"older regularity of the solution of the delay differential equation \eqref{eq:DiscDDE2}. We will use this theorem in the next section when proving error estimate for the randomized Euler algorithm. 
%Since we were not able to find references in literature that are  suitable for \eqref{eq:DiscDDE2} under the assumptions \ref{ass:A1}-\ref{ass:A4}, for the completeness and convenience of the reader we provide its justification.
%%%%%%%%%%%%%%%
\begin{thm}
\label{sol_dde_prop_1}
    Let $n\in\N\cup\{0\}$, $\tau\in (0,+\infty)$, $x_0\in\R^d$ and let $f$ satisfy assumptions \ref{ass:A1}-\ref{ass:A4}. Then there exists a unique absolutely continuous solution $x=x(x_0,f)$ to \eqref{eq:DiscDDE2}
    %, such that for $t\in [0,(n+1)\tau]$
    %\begin{equation}
    %    x(t)=\sum\limits_{j=-1}^n \phi_j(t-j\tau)\cdot\mathbf{1}_{[j\tau,(j+1)\tau]}(t),
    %\end{equation}
    such that for $j=0,1,\ldots,n$ we have
    \[
    \label{upper_est_Kj}
        \sup\limits_{0\leq t \leq \tau}\|\phi_j(t)\|\leq K_j
    \]
    where  $K_{-1}:=\|x_0\|$ and
    \begin{equation}
    \label{def_K_j}
        K_j=(1+K_{j-1})(1+\|K\|_{L^1([j\tau,(j+1)\tau])})\cdot\exp\Bigl((1+K_{j-1})\|K\|_{L^1([j\tau,(j+1)\tau])}\Bigr).
    \end{equation}
    Moreover, if we additionally assume that for some $p\in (1,+\infty]$ the function $K$ in \ref{ass:A3} satisfies
    \begin{enumerate}[label=\textbf{(A5)},ref=(A5)]
    \item\label{ass:A5} 
    \hspace{4cm}$K\in L^p([0,(n+1)\tau])$,
    \end{enumerate}
    then for all $j=0,1,\ldots,n$, $t,s\in [0,\tau]$ it holds
    \begin{equation}
    \label{phi_j_Lipschitz}
        \|\phi_j(t)-\phi_j(s)\|\leq (1+K_{j-1})(1+K_{j}) \|K\|_{L^p([j\tau,(j+1)\tau])}|t-s|^{1-\frac{1}{p}}.
    \end{equation}
\end{thm}
\begin{proof} We proceed by induction.    We start with the case when $j=0$ and consider the following initial-value problem
\begin{equation}
\label{eqODE_0}
    \begin{cases}
        \phi_0'(t)=g_0(t,\phi_0(t)), & t\in [0,\tau], \\
        \phi_0(0)=x_{0}, 
    \end{cases}
\end{equation}
with $g_0(t,x)=f(t,x,\phi_{-1}(t))=f(t,x,x_0)$. Of course for all $t\in [0,\tau]$ the function $g_0(t,\cdot)$ is continuous and for all $x\in\R^d$ the function $g_0(\cdot,x)$ is Borel measurable. Moreover, by \eqref{assumpt_a3} we have $\|g_0(t,x)\|\leq K(t)(1+\|x_0\|)(1+\|x\|)$ for all $(t,x)\in [a,b]\times\R^d$, and by \eqref{assumpt_a4} for every compact set $U$ in $\R^d$ there exists a positive function $L_U\in L^1([0,(n+1)\tau])$ such that for all $t\in [0,\tau]$, $x,y\in U$ it holds
\begin{equation}
    \|g_0(t,x)-g_0(t,y)\|\leq L_U(t) (1+\|x_0\|) \|x-y\|.
\end{equation}
Therefore, by Lemma \ref{lem_ode_1} we have that there exists a unique absolutely continuous solution $\phi_0:[0,\tau]\to\R^d$ of \eqref{eqODE_0} that satisfies \eqref{upper_est_Kj} with $j=0$. In addition, if $K\in L^p([0,(n+1)\tau])$ for some $p\in (1,+\infty]$ then by Lemma \ref{lem_ode_1} we obtain that $\phi_0$ satisfies \eqref{phi_j_Lipschitz} for $j=0$. 

Let us now assume that for some $j\in \{0,1,\ldots,n-1\}$ there exists a unique absolutely continuous solution $\phi_j:[0,\tau]\to\R^d$ of 
\begin{equation}
\label{eqODE_j}
\begin{cases}
\phi_j'(t)=g_j(t,\phi_j(t)), & t\in [0,\tau], \\
\phi_j(0)=\phi_{j-1}(\tau), 
\end{cases}
\end{equation}
where $g_j(t,x)=f(t+j\tau,x,\phi_{j-1}(t))$, and that satisfies \eqref{upper_est_Kj} with \eqref{phi_j_Lipschitz}, if $K\in L^p([0,(n+1)\tau])$ for some $p\in (1,+\infty]$. We consider the following initial-value problem
\begin{equation}
\label{eqODE_j_1}
\begin{cases}
\phi_{j+1}'(t)=g_{j+1}(t,\phi_{j+1}(t)), & t\in [0,\tau], \\
\phi_{j+1}(0)=\phi_j({\tau}), 
\end{cases}
\end{equation}
with $g_{j+1}(t,x)=f(t+(j+1)\tau,x,\phi_j(t))$. Since $\phi_j$ is continuous on $[0,\tau]$, it is straightforward to see that for all $t\in [0,\tau]$ the function $g_{j+1}(t,\cdot)$ is continuous and for all $x\in\R^d$ the function $g_{j+1}(\cdot,x)$ is Borel measurable. Moreover, by \eqref{upper_est_Kj} for all $(t,x)\in [0,\tau]\times\R^d$
\begin{eqnarray}
    &&\|g_{j+1}(t,x)\|\leq K(t+(j+1)\tau) (1+\|\phi_j(t)\|)(1+\|x\|)\notag\\
    &&\leq K(t+(j+1)\tau) (1+K_j)(1+\|x\|),
\end{eqnarray}
where
\begin{equation}
    \int\limits_0^{\tau} K(t+(j+1)\tau)\,\de t=\int\limits_{(j+1)\tau}^{(j+2)\tau}K(t)\,\de t\leq\int\limits_{0}^{(n+1)\tau}K(t)\,\de t<+\infty.
\end{equation}
Furthermore, by \eqref{assumpt_a4} and \eqref{upper_est_Kj} for every compact set $U$ in $\R^d$ there exists a positive function $L_U\in L^1([0,(n+1)\tau])$ such that for all $t\in [0,\tau]$, $x,y\in U$ it holds
\begin{equation}
    \|g_{j+1}(t,x)-g_{j+1}(t,y)\|\leq L_U(t+(j+1)\tau)(1+K_j)\|x-y\|,
\end{equation}
where
\begin{equation}
    \int\limits_0^{\tau} L_U(t+(j+1)\tau)\,\de t=\int\limits_{(j+1)\tau}^{(j+2)\tau}L_U(t)\,\de t\leq\int\limits_{0}^{(n+1)\tau}L_U(t)\,\de t<+\infty.
\end{equation}
Hence, by Lemma \ref{lem_ode_1} there exists a unique absolutely continuous solution $\phi_{j+1}:[0,\tau]\to\R^d$ of \eqref{eqODE_j_1}. By the inductive assumption and Lemma \ref{lem_ode_1} we get that
\begin{eqnarray}
&&\sup\limits_{0\leq t\leq \tau}\|\phi_{j+1}(t)\|\leq \Bigl(\|\phi_j(\tau)\|+(1+K_j)\int\limits_0^{\tau}K(t+(j+1)\tau)\,\de t\Bigr)\notag\\
&&\times\exp\Bigl((1+K_j)\int\limits_0^{\tau}K(t+(j+1)\tau)\,\de t\Bigr)\notag\\
&&\leq (1+K_j)(1+\|K\|_{L^1([(j+1)\tau,(j+2)\tau])})\exp\Bigl((1+K_j)\|K\|_{L^1([(j+1)\tau,(j+2)\tau])}\Bigr)=K_{j+1},\notag
\end{eqnarray}
and
\begin{equation}
    \|\phi_{j+1}(t)-\phi_{j+1}(s)\|\leq \bar K_{j+1} |t-s|^{1-\frac{1}{p}},
\end{equation}
where
\begin{eqnarray}
&&\bar K_{j+1}=(1+K_j)\Bigl(\int\limits_0^{\tau}|K(t+(j+1)\tau)|^p\,\de t\Bigr)^{1/p}\notag\\
&&\times\Biggl(1+\Bigl(\|\phi_j(\tau)\|+(1+K_j)\int\limits_0^{\tau} K(t+(j+1)\tau)\,\de t\Bigr)\exp\Bigl((1+K_j)\int\limits_0^{\tau} K(t+(j+1)\tau) \,\de t\Bigr)\Biggr)\notag\\
&&\leq (1+K_j)\|K\|_{L^p([(j+1)\tau,(j+2)\tau])}(1+K_{j+1}).\notag
\end{eqnarray}
This ends the inductive proof.
\end{proof}
%%%%%%%%%%%%%%%%%
\begin{rem}
  Theorem \ref{sol_dde_prop_1} can be applied, for example, to the function
  \begin{equation}
      f(t,x,z)=K(t) \cdot \cos(x^2) \cdot |z|^{\alpha}, \quad (t,x,z)\in [0,(n+1)\tau]\times\mathbb{R}\times\mathbb{R},
  \end{equation}
  where $K$ is any function from $L^1([0,(n+1)\tau])$ and $\alpha\in (0,1]$.
\end{rem}
%%%%%%%%%%%%%%%%%%%%%
\section{Error of the randomized  Euler scheme}

In this section we perform detailed  error analysis for the randomized Euler. As mentioned in Section 1, for the error analysis we impose global Lipschitz assumption on $f=f(t,x,z)$ with respect to $x$ together with global H\"older condition with respect to $z$. Namely, instead of \ref{ass:A3} and \ref{ass:A4}, we assume 
\begin{enumerate}[label=\textbf{(A3')},ref=(A3')]
    \item\label{ass:A3'}
    there exist $p\in [2,+\infty]$, $\alpha\in (0,1]$ and  $\bar K,L:[0,(n+1)\tau]\to [0,+\infty)$ such that $\bar K,L\in L^p([0,(n+1)\tau])$ and for all $t\in [0,(n+1)\tau]$
    \begin{equation}
        \|f(t,0,0)\|\leq \bar K(t),
    \end{equation}
     and for all $t\in [0,(n+1)\tau]$, $x_1,x_2,z_1,z_2\in\R^d$
    \begin{equation}
        \|f(t,x_1,z_1)-f(t,x_2,z_2)\|\leq L(t)\Bigl( \|x_1-x_2\|+\|z_1-z_2\|^{\alpha}\Bigr).
    \end{equation}
\end{enumerate}
%%%%%%%%%
\begin{rem}
\label{rem_str_assum}
Note that the assumptions \ref{ass:A1}, \ref{ass:A2}, \ref{ass:A3'} are stronger than the assumptions \ref{ass:A1}-\ref{ass:A4}. To see that note that if $f$ satisfies \ref{ass:A1}, \ref{ass:A2}, \ref{ass:A3'}  then we get  for all $t\in [0,(n+1)\tau]$ and $x,x_1,x_2,z\in\mathbb{R}^d$ that
\begin{equation}
    \|f(t,x,z)\|\leq (\bar  K(t)+L(t))(1+\|x\|)(1+\|z\|),
\end{equation}
and
\begin{equation}
    \|f(t,x_1,z)-f(t,x_2,z)\|\leq L(t)\|x_1-x_2\|,
\end{equation}
since $1+\|x\|+\|z\|\leq (1+\|x\|)(1+\|z\|)$ and $\|z\|^{\alpha}\leq 1+\|z\|$ for all $x,z\in\mathbb{R}^d$.  Hence, the assumptions \ref{ass:A1}-\ref{ass:A4} are satisfied with $K=\bar K+L\in L^p([0,(n+1)\tau])$, $L_U=L\in L^p([0,(n+1)\tau])$ for any compact set $U\subset\mathbb{R}^d$, and under the assumptions  \ref{ass:A1}, \ref{ass:A2}, \ref{ass:A3'} the thesis of Theorem \ref{sol_dde_prop_1} holds.
\end{rem}
The main result of this section is as follows.
%%%%%%%%%%%%%%%
\begin{thm} 
\label{rate_of_conv_expl_Eul} 
Let $n\in\N\cup\{0\}$, $\tau\in (0,+\infty)$, $x_0\in\R^d$, and let $f$ satisfy the assumptions  \ref{ass:A1}, \ref{ass:A2}, \ref{ass:A3'} for some $p\in [2,+\infty)$ and $\alpha\in (0,1]$. There exist $C_0,C_1,\ldots,C_n\in (0,+\infty)$ such that for  all $N\geq \lceil \tau\rceil$ and $j=0,1,\ldots,n$ it holds
	\begin{equation}
	\label{error_main_thm}
		\Bigl\|\max\limits_{0\leq i\leq N}\|x(t_i^j)-y_i^j\|\Bigl\|_{L^p(\Omega)}\leq C_j h^{\frac{1}{2}\alpha^j}.
	\end{equation}
In particular, if $\alpha=1$ then for $j=0,1,\ldots,n$
	\begin{equation}\label{error_main_thm_alpha1}
		\Bigl\|\max\limits_{0\leq i\leq N}\|x(t_i^j)-y_i^j\|\Bigl\|_{L^p(\Omega)}\leq C_j h^{1/2}.
	\end{equation}
\end{thm}
\begin{proof} In the proof  we use the following auxiliary notation: $\alpha_{k}^{j}=k+\gamma_{k+1}^j$ and $\delta_{k+1}^j=h\cdot \alpha_{k}^{j}$. Then $\delta_{k+1}^j$ is uniformly distributed in $(t_k^0,t_{k+1}^0)$ and $\theta_{k+1}^j=\delta_{k+1}^j+j\tau$ is uniformly distributed in $(t_k^j,t_{k+1}^j)$.

    We start with $j=0$ and consider the initial-vale problem \eqref{eqODE_0}. We define the auxiliary randomized Euler scheme as 
    \begin{eqnarray}
        \label{aux_euler_1}
	        &&\bar y_0^0= y^0_0=y^{-1}_N=x_0,\\
        \label{aux_euler_11}	
	        &&\bar y_{k+1}^0=\bar y_k^0+h\cdot g_0(\theta_{k+1}^0,\bar y_k^0), \ k=0,1,\ldots, N-1.
\end{eqnarray}
Since $\bar y_0^0=y_0^0=x_0$ and $g_0(t,x)=f(t,x,x_0)$, for all $k=0,\ldots,N$ we have that $\bar y_k^0=y_k^0$. Moreover, by (A1), (A2) and (A3') we have that $g_0$ is Borel measurable, and for all $t\in [0,\tau]$ and  $x,y\in\R^d$ 
    \begin{eqnarray}
        &&\|g_0(t,x)\|\leq K(t)(1+\|x_0\|)(1+\|x\|),\notag\\
        &&\|g_0(t,x)-g_0(t,y)\|\leq L(t)\|x-y\|,
    \end{eqnarray}
    where $K=\bar K+L$, as stated in Remark \ref{rem_str_assum}.
    Hence, by Theorem \ref{sol_dde_prop_1} and by using analogous arguments as in the proof of Theorem 4.3 in \cite{RKYW2017} we get
\begin{eqnarray}
   && \Bigl\|\max\limits_{0\leq i\leq N}\|\phi_0(t_i^0)-y_i^0\|\Bigl\|_{L^p(\Omega)}\leq 2^{1-\frac{1}{p}}\exp\Bigl(\frac{(2\tau)^{p-1}}{p}\|L\|^p_{L^p([0,\tau])}\Bigr)\notag\\
   &&\times (1+K_{-1})(1+K_0)\|K\|_{L^p([0,\tau])}\Bigl(2C_p\tau^{\frac{p-2}{2p}}+\tau^{1-\frac{1}{p}}\|L\|_{L^p([0,\tau])}\Bigr)h^{1/2}=C_0h^{1/2},\notag
\end{eqnarray}
where $C_0$ does not depend on $N$. Since $\phi_0(t_i^0)=x(t_i^0)$ we get \eqref{error_main_thm} for $j=0$.

Let us now assume that there exists $l\in\{0,1,\ldots,n-1\}$ for which there exists $C_l \in (0,+\infty)$ such that for all $N\geq \lceil \tau\rceil$
\begin{equation}
\label{ind_assumpt_1}
    \Bigl\|\max\limits_{0\leq i\leq N}\|\phi_l(t_i^0)-y_i^l\|\Bigl\|_{L^p(\Omega)}\leq C_l h^{\frac{1}{2}\alpha^l}.
\end{equation}
We consider the following initial-value problem
\begin{equation}
\label{eqODE_l_1}
\begin{cases}
\phi_{l+1}'(t)=g_{l+1}(t,\phi_{l+1}(t)), & t\in [0,\tau], \\
\phi_{l+1}(0)=\phi_l({\tau}), 
\end{cases}
\end{equation}
with $g_{l+1}(t,x)=f(t+(l+1)\tau,x,\phi_l(t))$. 
Recall that by (A1), (A2), (A3'), Theorem \ref{sol_dde_prop_1}, and  Remark \ref{rem_str_assum}  the function $g_{l+1}$ is Borel measurable and for all $t\in [0,\tau]$, $x,y\in\R^d$ it satisfies
\begin{equation}
  \|g_{l+1}(t,x)\|\leq K(t+(l+1)\tau)(1+K_l)(1+\|x\|),
\end{equation}
\begin{equation}
    \|g_{l+1}(t,x)-g_{l+1}(t,y)\|\leq L(t+(l+1)\tau)\|x-y\|.
\end{equation}
We define the auxiliary randomized Euler scheme as follows
\begin{eqnarray}
        \label{aux_euler_l1}
	       &&\bar y_0^{l+1}= y^{l+1}_0=y^{l}_N,\\
        %\label{aux_euler_11}	
	        &&\bar y_{k+1}^{l+1}=\bar y_k^{l+1}+h\cdot g_{l+1}(\alpha_k^{l+1}\cdot h,\bar y_k^{l+1}), \ k=0,1,\ldots, N-1.
\end{eqnarray}
From the definition it follows that $\bar y_k^j$, $j=0,1,\ldots,n$, $k=0,1,\ldots,N$, is measurable with respect to the $\sigma$-filed generated by \eqref{sig_filed1}, so as $y_k^j$. Moreover,  $\bar y^{l+1}_i$ approximates $\phi_{l+1}$ at $t_i^0$, however $\{\bar y_i^{l+1}\}_{i=0,1,\ldots,N}$ is not implementable. We use $\bar y^{l+1}_i$ only in order to estimate the error \eqref{error_main_thm} of $y^{l+1}_i$, since  it holds
\begin{equation}\label{error_decomp}
\begin{aligned}
\Bigl\|\max\limits_{0\leq i\leq N}\|\phi_{l+1}(t_i^0)-y_i^{l+1}\|\Bigl\|_{L^p(\Omega)} &\leq \Bigl\|\max\limits_{0\leq i\leq N}\|\phi_{l+1}(t_i^0)-\bar y_i^{l+1}\|\Bigl\|_{L^p(\Omega)} \\
&\quad+\Bigl\|\max\limits_{0\leq i\leq N}\|\bar y_i^{l+1}-y_i^{l+1}\|\Bigl\|_{L^p(\Omega)}.
\end{aligned}
\end{equation}
Firstly, we estimate $\displaystyle{\Bigl\|\max\limits_{0\leq i\leq N}\|\bar y_i^{l+1}-y_i^{l+1}\|\Bigl\|_{L^p(\Omega)}}$. For $k\in\{1,\ldots,N\}$ we get
\begin{align*}
\bar y^{l+1}_k-y^{l+1}_k &= \sum\limits_{j=1}^k(\bar y^{l+1}_j-\bar y^{l+1}_{j-1})-\sum\limits_{j=1}^k(y^{l+1}_j-y^{l+1}_{j-1})\notag\\
&=h\sum\limits_{j=1}^k\Bigl(f(\theta_j^{l+1},\bar y^{l+1}_{j-1},\phi_{l}(\alpha_{j-1}^{l+1}\cdot h))-f(\theta_j^{l+1},y^{l+1}_{j-1},y^l_{j-1})\Bigr),
\end{align*}
which gives
\[
\|\bar y^{l+1}_k-y^{l+1}_k\|\leq h\sum\limits_{j=1}^kL(\theta_j^{l+1})\cdot \|\bar y^{l+1}_{j-1}-y^{l+1}_{j-1}\|+h\sum\limits_{j=1}^k L(\theta_j^{l+1})\cdot \|\phi_{l}(\alpha_{j-1}^{l+1}\cdot h)-y^{l}_{j-1}\|^{\alpha}.
\]
Note that the random variables $L(\theta_j^{l+1})$ and $\|\phi_{l}(\alpha_{j-1}^{l+1}\cdot h)-y^{l}_{j-1}\|^{\alpha}$ are not independent. However,  by Theorem \ref{sol_dde_prop_1}  we get
\begin{align*}
\|\phi_{l}(\alpha_{j-1}^{l+1}\cdot h)-y^{l}_{j-1}\| &\leq
\|\phi_{l}(\alpha_{j-1}^{l+1}\cdot h)-\phi_{l}(t_{j-1}^0)\|+\|\phi_{l}(t_{j-1}^0)-y^{l}_{j-1}\| \\
&\leq (1+K_{l-1})(1+K_l)\|K\|_{L^p([l\tau,(l+1)\tau])}\cdot h^{1-\frac{1}{p}}+\|\phi_l(t_{j-1}^0)-y_{j-1}^l\|.
\end{align*}
Therefore for $k\in\{1,\ldots,N\}$
\begin{align*}
\|\bar y^{l+1}_k-y^{l+1}_k\| &\leq h\sum\limits_{j=1}^k L(\theta_j^{l+1})\cdot\max\limits_{0\leq i\leq j-1} \|\bar y^{l+1}_i-y^{l+1}_i\| \\
&\quad +h\Bigl(\sum\limits_{j=1}^k L(\theta_j^{l+1})\Bigr)\cdot\Bigl(c_{l+1}h^{\alpha(1-\frac{1}{p})}+\max\limits_{0\leq i\leq N}\|\phi_{l}(t_i^0)-y_i^l\|^{\alpha}\Bigr).
\end{align*}
Since $\|\bar y^{l+1}_0-y^{l+1}_0\|=0$ and due to the fact that the random variables $\Bigl(\sum\limits_{j=1}^k L(\theta_j^{l+1})\Bigr)$, $\max\limits_{0\leq i\leq N} \|\phi_l(t_i^0)-y^{l}_i\|^{\alpha}$ are independent, we get
\[
\begin{split}
\mathbb{E}\Bigl(\max\limits_{0\leq i\leq k}\|\bar y^{l+1}_i-y^{l+1}_i\|^p\Bigr)\leq c_p h^p\mathbb{E}\Bigl[\sum\limits_{j=1}^k L(\theta_j^{l+1})\cdot\max\limits_{0\leq i\leq j-1} \|\bar y^{l+1}_i-y^{l+1}_i\|\Bigr]^p \\
+\tilde c_p\mathbb{E}\Bigl[h\sum\limits_{j=1}^k L(\theta_j^{l+1})\Bigr]^p\cdot \Bigl(\tilde c_{l+1}h^{\alpha(p-1)}+\mathbb{E}\Bigl[\max\limits_{0\leq i\leq N}\|\phi_{l}(t_i^0)-y_i^l\|^{\alpha p}\Bigr]\Bigr).
\end{split}
\]
By Jensen inequality and \eqref{ind_assumpt_1} we have
\begin{equation}
\label{est_1}
    \mathbb{E}\Bigl[\max\limits_{0\leq i\leq N}\|\phi_{l}(t_i^0)-y_i^l\|^{\alpha p}\Bigr]\leq\Biggl( \mathbb{E}\Bigl[\max\limits_{0\leq i\leq N}\|\phi_{l}(t_i^0)-y_i^l\|^{p}\Bigr]\Biggr)^{\alpha}\leq C_l^ph^{\frac{p}{2}\alpha^l}.
\end{equation}
Since the random variables $L(\theta_j^{l+1})$ and $\max\limits_{0\leq i \leq j-1}\|\bar y_i^{l+1}-y_i^{l+1}\|$ are independent, $\theta_j^{l+1}\sim U(t_{j-1}^{l+1},t_{j}^{l+1})$, we get by the H\"older inequality
\begin{equation}
\label{est_2}
    h^p\mathbb{E}\Bigl[\sum\limits_{j=1}^k L(\theta_j^{l+1})\cdot\max\limits_{0\leq i\leq j-1} \|\bar y^{l+1}_i-y^{l+1}_i\|\Bigr]^p\leq \tau^{p-1}\sum\limits_{j=1}^k\int\limits_{t_{j-1}^{l+1}}^{t_{j}^{l+1}}(L(t))^p \de t\cdot\mathbb{E}\Bigl[\max\limits_{0\leq i\leq j-1}\|\bar y_i^{l+1}-y_i^{l+1}\|^p\Bigr],
\end{equation}
and
\begin{equation}
\label{est_3}
    h^p\cdot\mathbb{E}\Bigl(\sum\limits_{j=1}^k L(\theta_j^{l+1})\Bigr)^p\leq\tau^{p-1}\int\limits_{(l+1)\tau}^{(l+2)\tau}(L(t))^p\,\de t<+\infty.
\end{equation}
Combining \eqref{est_1}, \eqref{est_2},  \eqref{est_3}, and using again the fact that $\bar y_0^{l+1}=y_0^{l+1}$ we arrive at
\begin{displaymath}
    \mathbb{E}\Bigl(\max\limits_{0\leq i\leq k}\|\bar y^{l+1}_i-y^{l+1}_i\|^p\Bigr)\leq
    \bar C_{2,l+1}h^{\frac{p}{2}\alpha^{l+1}}+\bar C_{1,l+1} \sum\limits_{j=1}^{k-1}\int\limits_{t_{j}^{l+1}}^{t_{j+1}^{l+1}}(L(t))^p \de t\cdot\mathbb{E}\Bigl[\max\limits_{0\leq i\leq j}\|\bar y_i^{l+1}-y_i^{l+1}\|^p\Bigr],
\end{displaymath}
for $k\in\{1,2,\ldots,N\}$. By applying weighted Gronwall's lemma (see, for example, Lemma 2.1 in \cite{RKYW2017}) we get that
\begin{equation*}
    \mathbb{E}\Bigl[\max\limits_{0\leq i\leq N}\|\bar y^{l+1}_i-y^{l+1}_i\|^p\Bigr]\leq \bar C_{3,l+1}\|L\|^p_{L^p([(l+1)\tau,(l+2)\tau])}e^{\bar C_{4,l+1}\|L\|^p_{L^p([(l+1)\tau,(l+2)\tau])}}h^{\frac{p}{2}\alpha^{l+1}},
\end{equation*}
and hence
\begin{equation}
\label{est_diff_bykyk}
    \Bigl\|\max\limits_{0\leq i\leq N}\|\bar y^{l+1}_i-y^{l+1}_i\|\Bigl\|_{L^p(\Omega)}\leq \bar C_{5,l+1}h^{\frac{1}{2}\alpha^{l+1}}.
\end{equation}
We now establish an upper bound on  $\displaystyle{\Bigl\|\max\limits_{0\leq i\leq N}\|\phi_{l+1}(t_i^0)-\bar y_i^{l+1}\|\Bigl\|_{L^p(\Omega)}}$. For $k\in\{1,2,\ldots,N\}$ we have
\begin{eqnarray}
        &&\phi_{l+1}(t_k^0)-\bar y^{l+1}_k=\phi_{l+1}(0)-\bar y_0^{l+1}+(\phi_{l+1}(t_k^0)-\phi_{l+1}(t_0^0))+(\bar y^{l+1}_k-\bar y^{l+1}_0)\notag\\
        &&=(\phi_l(t_N^0)-y_N^l)+\sum\limits_{j=1}^k(\phi_{l+1}(t_j^0)-\phi_{l+1}(t_{j-1}^0))-\sum\limits_{j=1}^k(\bar y^{l+1}_j-\bar y^{l+1}_{j-1})\notag\\
        &&=(\phi_l(t_N^0)-y_N^l)+\sum\limits_{j=1}^k\Bigl(\int\limits_{t_{j-1}^0}^{t_{j}^0}g_{l+1}(s,\phi_{l+1}(s))\de s-h\cdot  g_{l+1}(\delta_j^{l+1},\bar y^{l+1}_{j-1})\Bigr)\notag\\
        \label{err_phi_byk_decomp1}
        &&=(\phi_l(t_N^0)-y_N^l)+S^k_{1,l+1}+S^k_{2,l+1}+S^k_{3,l+1},
\end{eqnarray}
where
\begin{align*}
S^k_{1,l+1} &= \sum\limits_{j=1}^k \Bigl(\int\limits_{t_{j-1}^0}^{t_{j}^0}g_{l+1}(s,\phi_{l+1}(s))\de s-h\cdot  g_{l+1}(\delta_j^{l+1},\phi_{l+1}(\delta_j^{l+1}))\Bigr), \\
S^k_{2,l+1} &= h\cdot\sum\limits_{j=1}^k\Bigl(g_{l+1}(\delta_j^{l+1},\phi_{l+1}(\delta_j^{l+1}))-g_{l+1}(\delta_j^{l+1},\phi_{l+1}(t_{j-1}^0))\Bigr), \\
S^k_{3,l+1} &= h\cdot\sum\limits_{j=1}^k\Bigl(g_{l+1}(\delta_j^{l+1},\phi_{l+1}(t_{j-1}^0))-g_{l+1}(\delta_j^{l+1},\bar y^{l+1}_{j-1})\Bigr).
\end{align*}
Since the function $[0,\tau]\ni t\mapsto g_{l+1}(t,\phi_{l+1}(t))$ is Borel measurable and, by Theorem \ref{sol_dde_prop_1} above,
\begin{equation}
    \|g_{l+1}(\cdot,\phi_{l+1}(\cdot))\|_{L^p([0,\tau])}\leq (1+K_l)(1+K_{l+1})\|K\|_{L^p([(l+1)\tau,(l+2)\tau])}<+\infty,
\end{equation}
we get by Theorem 3.1 in \cite{RKYW2017} that
\begin{equation}
\label{est_S1}
    \Bigl\|\max\limits_{1\leq k\leq N}\|S^k_{1,l+1}\|\Bigl\|_{L^p(\Omega)}\leq 2C_p\tau^{\frac{p-2}{2p}} (1+K_l)(1+K_{l+1})\|K\|_{L^p([(l+1)\tau,(l+2)\tau])}\cdot h^{1/2}.
\end{equation}
By Theorem \ref{sol_dde_prop_1} we get
\begin{align*}
        &\|S^k_{2,l+1}\|\leq h\sum\limits_{j=1}^k L(\theta_j^{l+1})\cdot\|\phi_{l+1}(\delta_j^{l+1})-\phi_{l+1}(t_{j-1}^0)\| \\
        &\leq h^{2-\frac{1}{p}}(1+K_l)(1+K_{l+1})\|K\|_{L^p[(l+1)\tau,(l+2)\tau]}\sum\limits_{j=1}^N L(\theta_j^{l+1}), 
\end{align*}
and by the H\"older inequality
\begin{equation}
\label{est_S2}
\begin{aligned}
    \Bigl\|\max\limits_{1\leq k\leq N}\|S^k_{2,l+1}\|\Bigl\|_{L^p(\Omega)}&\leq h^{1-\frac{1}{p}}\tau^{1-\frac{1}{p}}(1+K_l)(1+K_{l+1})\|K\|_{L^p[(l+1)\tau,(l+2)\tau]} \\
    &\times\mathbb{E}\Bigl[h\cdot\sum\limits_{j=1}^N (L(\theta_j^{l+1}))^p\Bigr] \\
    \leq h^{1-\frac{1}{p}}\tau^{1-\frac{1}{p}}&(1+K_l)(1+K_{l+1})\|K\|_{L^p[(l+1)\tau,(l+2)\tau]}\|L\|_{L^p[(l+1)\tau,(l+2)\tau]}.
\end{aligned}
\end{equation}
Moreover,
\begin{equation}
\label{est_S3}
\begin{aligned}
   \|S^k_{3,l+1}\|&\leq h\cdot\sum\limits_{j=1}^kL(\theta^{l+1}_j)\cdot\|\phi_{l+1}(t_{j-1}^0)-\bar y^{l+1}_{j-1}\| \\
   \leq hL(\theta_1^{l+1})\cdot\|\phi_l(t_N^0)-y_N^l\|&+h\cdot\sum\limits_{j=2}^kL(\theta_j^{l+1})\cdot\max\limits_{0\leq i \leq j-1}\|\phi_{l+1}(t_i^0)-\bar y^{l+1}_i\|.
\end{aligned}
\end{equation}
Hence, from \eqref{err_phi_byk_decomp1} and \eqref{est_S3} we have for $k\in\{1,2,\ldots,N\}$
\begin{align*}
    \max\limits_{0\leq i\leq k}\|\phi_{l+1}(t_i^0)-\bar y^{l+1}_i\|&\leq (1+h L(\theta_1^{l+1}))\|\phi_l(t_N^0)-y_N^l\|+\max\limits_{1\leq k\leq N}\|S^k_{1,l+1}\| \\
    +\max\limits_{1\leq k\leq N}\|S^k_{2,l+1}\|&+h\cdot\sum\limits_{j=2}^kL(\theta_j^{l+1})\cdot\max\limits_{0\leq i \leq j-1}\|\phi_{l+1}(t_i^0)-\bar y^{l+1}_i\|.
\end{align*}
Since the random variables $L(\theta_1^{l+1})$ and $\|\phi_l(t_N^0)-y_N^l\|$ are independent, we obtain
\begin{align*}
    &\mathbb{E}\Bigl[\max\limits_{0\leq i\leq k}\|\phi_{l+1}(t_i^0)-\bar y^{l+1}_i\|^p\Bigr]\leq c_p\mathbb{E}(1+h L(\theta_1^{l+1}))^p\cdot\mathbb{E}\|\phi_l(t_N^0)-y_N^l\|^p \\
    &+ c_p\Bigl(\mathbb{E}\Bigl[\max\limits_{1\leq k\leq N}\|S^k_{1,l+1}\|^p\Bigr]+\mathbb{E}\Bigl[\max\limits_{1\leq k\leq N}\|S^k_{2,l+1}\|^p\Bigr]\Bigr) \\
    &+c_p\mathbb{E}\Bigl[h\cdot\sum\limits_{j=2}^kL(\theta_j^{l+1})\cdot\max\limits_{0\leq i \leq j-1}\|\phi_{l+1}(t_i^0)-\bar y^{l+1}_i\|\Bigr]^p.
\end{align*}
Moreover
\begin{displaymath}
    \mathbb{E}\left[1+h L(\theta_1^{l+1})\right]^p\leq c_p\Bigl(1+(b-a)^{p-1}\cdot\|L\|^p_{L^p[(l+1)\tau,(l+2)\tau]}\Bigr)<+\infty,
\end{displaymath}
and by the H\"older inequality, and the fact that the random variables $L(\theta_j^{l+1})$ and $\max\limits_{0\leq i \leq j-1}\|\phi_{l+1}(t_i^0)-\bar y_i^{l+1}\|$ are independent we have
\begin{align*}
    &\mathbb{E}\Bigl[h\cdot\sum\limits_{j=2}^kL(\theta_j^{l+1})\cdot\max\limits_{0\leq i \leq j-1}\|\phi_{l+1}(t_i^0)-\bar y^{l+1}_i\|\Bigr]^p \\
    &\leq h^p\cdot N^{p-1}\cdot\sum\limits_{j=2}^k\mathbb{E}(L(\theta_j^{l+1}))^p\cdot\mathbb{E}\Bigl[\max\limits_{0\leq i \leq j-1}\|\phi_{l+1}(t_i^0)-\bar y^{l+1}_i\|^p\Bigr] \\
    &\leq \tau^{p-1}\sum\limits_{j=1}^{k-1}\int\limits_{t_j^{l+1}}^{t_{j+1}^{l+1}}(L(t))^p \de t\cdot\mathbb{E}\Bigl[\max\limits_{0\leq i \leq j}\|\phi_{l+1}(t_i^0)-\bar y^{l+1}_i\|^p\Bigr].
\end{align*}
Therefore, from  for all $k\in\{1,2,\ldots,N\}$ the following inequality holds
\begin{align*}
\mathbb{E}\Bigl[\max\limits_{0\leq i\leq k}\|\phi_{l+1}(t_i^0)-\bar y^{l+1}_i\|^p\Bigr] &\leq \tilde c_p\Bigl(\mathbb{E}\|\phi_l(t_N^0)-y_N^l\|^p \\
&\quad + \mathbb{E}\Bigl[\max\limits_{1\leq k\leq N}\|S^k_{1,l+1}\|^p\Bigr]+\mathbb{E}\Bigl[\max\limits_{1\leq k\leq N}\|S^k_{2,l+1}\|^p\Bigr]\Bigr) \\
&\quad +c_p\tau^{p-1}\sum\limits_{j=1}^{k-1}\int\limits_{t_j^{l+1}}^{t_{j+1}^{l+1}}(L(t))^p \de t\cdot\mathbb{E}\Bigl[\max\limits_{0\leq i \leq j}\|\phi_{l+1}(t_i^0)-\bar y^{l+1}_i\|^p\Bigr].
\end{align*}
By using  Gronwall's lemma (see,  Lemma 2.1 in \cite{RKYW2017}), \eqref{ind_assumpt_1}, \eqref{est_S1}, and \eqref{est_S2}  we get for all $k\in\{1,2,\ldots,N\}$
\begin{align*}
\mathbb{E}\Bigl[\max\limits_{0\leq i\leq k}\|\phi_{l+1}(t_i^0)-\bar y^{l+1}_i\|^p\Bigr] & \leq \tilde c_p\Bigl(\mathbb{E}\|\phi_l(t_N^0)-y_N^l\|^p\\
+\mathbb{E}\Bigl[\max\limits_{1\leq k\leq N}\|S^k_{1,l+1}\|^p\Bigr]&+\mathbb{E}\Bigl[\max\limits_{1\leq k\leq N}\|S^k_{2,l+1}\|^p\Bigr]\Bigr)\exp\Bigl(c_p\tau^{p-1}\sum\limits_{j=1}^{k-1}\int\limits_{t_j^{l+1}}^{t_{j+1}^{l+1}}(L(t))^p\de t\Bigr) \\
&\leq C_{l+1}\cdot\exp\Bigl((c_p\tau^{p-1}\|L\|_{L^p[(l+1)\tau,(l+2)\tau]}\Bigr)\cdot h^{\frac{p}{2}\alpha^l},
\end{align*}
which gives
\begin{equation}
\label{error_phi_byk2}
    \Bigl\|\max\limits_{0\leq i\leq N}\|\phi_{l+1}(t_i^0)-\bar y_i^{l+1}\|\Bigl\|_{L^p(\Omega)}\leq C_{l+1} h^{\frac{1}{2}\alpha^l}.
\end{equation}
Combining \eqref{error_decomp}, \eqref{est_diff_bykyk}, and \eqref{error_phi_byk2} we finally obtain
\begin{equation}
     \Bigl\|\max\limits_{0\leq i\leq N}\|\phi_{l+1}(t_i^0)- y_i^{l+1}\|\Bigl\|_{L^p(\Omega)}\leq C_{l+1} h^{\frac{1}{2}\alpha^{l+1}},
\end{equation}
which ends the inductive part of the proof. Finally,  $\phi_{l+1}(t_i^0)=\phi_{l+1}(ih)=x(ih+(l+1)\tau)=x(t_i^{l+1})$ and the proof of \eqref{error_main_thm} is finished. 
\end{proof}
%%%%%%%%%%%%%%%
From Theorem \ref{rate_of_conv_expl_Eul} we see that in the case when the horizon parameter $n$ is fixed and $\alpha=1$ the randomized Euler scheme recovers the classical optimal  convergence rate for Monte-Carlo methods, since its error is $O(N^{-1/2})$, in the whole time interval $[0,(n+1)\tau]$, and uses $O(N)$ values of $f$, see \cite{novak1}. For $\alpha\in (0,1)$ we see that the upper bound on the error increases from interval to interval, reaching $O(N^{-\frac{1}{2}\alpha^n})$ in the final time point $(n+1)\tau$. This error behavior seems to be specific for DDEs under lack of global Lipschitz assumption with respect to $z$, see, for example, \cite{CZPMPP}.

In the case when $\max\{\|L\|_{L^{\infty}([0,(n+1)\tau])},\|K\|_{L^{\infty}([0,(n+1)\tau])}\}<+\infty$ and $\alpha=1$ we can establish for the randomized Euler scheme convergence with probability one.
%%%%%%%%
    \begin{prop} 
    \label{as_conv}
    Let $n\in\N\cup\{0\}$, $\tau\in (0,+\infty)$, $x_0\in\R^d$, and let $f$ satisfy the assumptions  \ref{ass:A1}, \ref{ass:A2}, \ref{ass:A3'} with $p=+\infty$, $\alpha=1$. Then for all $\varepsilon\in(0,1/2)$ there exists $\tilde \eta_{\varepsilon,n}\in\bigcap\limits_{q\in [2,+\infty)}L^q(\Omega)$ such that
    \begin{equation}
        \max\limits_{0\leq j\leq n}\max\limits_{0\leq i\leq N}\|x(t_i^j)-y_i^j\|\leq\tilde\eta_{\varepsilon,n}\cdot N^{-\frac{1}{2}+\varepsilon} \quad \hbox{almost surely}
    \end{equation}
    for all  $N\geq\lceil \tau\rceil$.
    \end{prop}
\begin{proof}
    Note that if $f$ satisfy the assumptions  \ref{ass:A1}, \ref{ass:A2}, \ref{ass:A3'} with $p=+\infty$, $\alpha=1$, then from the proof of Theorem \ref{rate_of_conv_expl_Eul} we get that for all $q\in [2,+\infty)$ there exist $C_0(q),C_1(q),\ldots,C_n(q)\in (0,+\infty)$ such that  for  all $N\geq \lceil \tau\rceil$ and $j=0,1,\ldots,n$ we have
	\begin{equation}
	\label{error_main_thm_alpha1q}
		\Bigl\|\max\limits_{0\leq i\leq N}\|x(t_i^j)-y_i^j\|\Bigl\|_{L^q(\Omega)}\leq C_j(q) h^{1/2}.
	\end{equation}
	Hence, from Lemma 2.1. in \cite{KloNeu2007} we have that for all $\varepsilon\in (0,1/2)$ and  there exist non-negative random variables $\eta_{\varepsilon,0},\eta_{\varepsilon,1},\ldots\eta_{\varepsilon,n}$ such that for all $j=0,1,\ldots,n$
	\begin{equation}
	    \max\limits_{0\leq i\leq N}\|x(t_i^j)-y_i^j\|\leq \eta_{\varepsilon,j}\cdot N^{-\frac{1}{2}+\varepsilon} \quad \hbox{almost surely}
	\end{equation}
	for all $N\geq\lceil \tau\rceil$. This implies the thesis with $\tilde \eta_{\varepsilon,n}=\max\limits_{0\leq j\leq n}\eta_{\varepsilon,j}$.
\end{proof}
%%%%%%%%%%%%%%%%%%%
\section{Numerical experiments}
In order to illustrate our theoretical findings we perform several numerical experiments. We chose the following exemplary right-hand side functions. The first, which satisfies the assumptions \ref{ass:A1}, \ref{ass:A2}, \ref{ass:A3'}
\begin{equation}\label{eq:f6}
    f_1(t,x,z)=k(t) \Bigl( x +0.01|z|^{\alpha}+ \sin (M x) \cdot \cos(P|z|^{\alpha} ) \Bigr),
\end{equation}
and the second, for which the assumption \ref{ass:A3'} is not satisfied globally, 
\begin{equation}\label{eq:f1}
    f_2(t,x,z)=k(t)\cdot\sin(10x)\cdot P|z|^{\alpha}/M,
\end{equation}
where $k$ is the following periodic function
\begin{equation}
    k(t)=\sum_{j=0}^n \Bigl((j+1)\tau-t\Bigr)^{-1/\gamma}\cdot\mathbf{1}_{[j\tau,(j+1)\tau]}(t),
\end{equation}
which belongs to $L^p([0,(n+1)\tau])$ for $\gamma> p$. Note that due to the $k$ function the problem could be stiff if we are close to the asymptotic. Actually, it could be stiff-nonstiff according to derivative sign. \\

In what follows we provide numerical evidence of theoretical results from Theorem \ref{rate_of_conv_expl_Eul}. In particular, we implement randomized Euler scheme \eqref{expl_euler_1}-\eqref{expl_euler_11} using Python programming language. Moreover, since for the right-hand side functions \eqref{eq:f6}, \eqref{eq:f1} we do not know the exact solution $x=x(t)$, we  approximate the mean square error  \eqref{eq:L2err} with
\[
\max\limits_{0\leq j\leq n}\left\|\max_{0\leq i\leq N}|\tilde y_i^j-y_i^j|\right\|_{L^2(\Omega)}\approx\max\limits_{0\leq j\leq n}\left(\frac1K\sum_{k=1}^{K}\max_{0\leq i\leq N}|\tilde y_i^j(\omega_k)-y_i^j(\omega_k)|^2\right)^\frac12,
\]
where $K\in\mathbb{N}$, $y_i^j$ is the output of the randomized Euler scheme on the initial mesh $t_i^j\udef j\tau+ih$ and $h\udef\frac{\tau}{N}$ for $i=0,\ldots,N-1$, while $\tilde y_i^j$ is the  reference solution obtained also from the randomized Euler scheme but on the refined mesh $\tilde t_i^j\udef j\tau+i\tilde h$ and $\tilde h\udef\frac{h}{m}=\frac{\tau}{mN}$ for $i=0,\ldots,mN-1$. \\

%About randomized Euler scheme applied to DDEs, we refer to \cite{asaiKloeden,JentzenNeuenkirch}.

\begin{exm}
\normalfont
In the following numerical tests we use \eqref{eq:f6} with parameters $M=10, P=100$. 

We fix the number of experiments $K=1000$ for each $N=10^l$, $l=1,\ldots,5$, and the reference solution is computed with stepsize $10^{-7}$; also, the horizon parameter is $n=5$.
We get the following results for $\gamma=2.1$: \\
letting $\alpha=0.1$, the negative mean square error slopes are
$0.46$, $0.40$, $0.41$ and $0.37$. See Figure \ref{fig:M10,K100,alpha0.1,gamma2.1,f6}; \\
letting $\alpha=0.5$, the negative mean square error slopes are
$0.44$, $0.44$, $0.39$ and $0.33$. See Figure \ref{fig:M10,K100,alpha0.5,gamma2.1,f6}; \\
letting $\alpha=1$,  the negative mean square error slopes are
$0.43$, $0.39$, $0.40$ and $0.37$. See Figure \ref{fig:M10,K100,alpha1,gamma2.1,f6}; \\
while, for $\gamma=5$: \\
letting $\alpha=0.1$, the negative mean square error slopes are
$0.78$, $0.75$, $0.71$ and $0.69$. See Figure \ref{fig:M10,K100,alpha0.1,gamma5,f6}; \\
letting $\alpha=0.5$, the negative mean square error slopes are
$0.78$, $0.74$, $0.71$ and $0.69$. See Figure \ref{fig:M10,K100,alpha0.5,gamma5,f6}; \\
letting $\alpha=1$, the negative mean square error slopes are
$0.78$, $0.75$, $0.73$ and $0.72$. See Figure \ref{fig:M10,K100,alpha1,gamma5,f6}.
% We fix the number of experiments $K=1000$ for each $N=10^l$, $l=1,\ldots,5$, and the reference solution is computed using $m=1000$; also, the horizon parameter is $n=5$.
% We get the following results for $\gamma=2.1$: \\
% letting $\alpha=0.1$, mean square error slope is
% $-0.53634914$. See Figure \ref{fig:M10,K100,alpha0.1,gamma2.1,f6}; \\
% letting $\alpha=0.5$, mean square error slope is
% $-0.54725184$. See Figure \ref{fig:M10,K100,alpha0.5,gamma2.1,f6}; \\
% letting $\alpha=1$, mean square error slope is
% $-0.52244241$. See Figure \ref{fig:M10,K100,alpha1,gamma2.1,f6}; \\
% while, for $\gamma=5$: \\
% letting $\alpha=0.1$, mean square error slope is $-0.79816347$. See Figure \ref{fig:M10,K100,alpha0.1,gamma5,f6}; \\
% letting $\alpha=0.5$, mean square error slope is $-0.79859262$. See Figure \ref{fig:M10,K100,alpha0.5,gamma5,f6}; \\
% letting $\alpha=1$, mean square error slope is $-0.7968649$. See Figure \ref{fig:M10,K100,alpha1,gamma5,f6}.
\begin{figure}
\centering
\begin{subfigure}{0.32\textwidth}
    \centering
    \includegraphics[width=\textwidth]{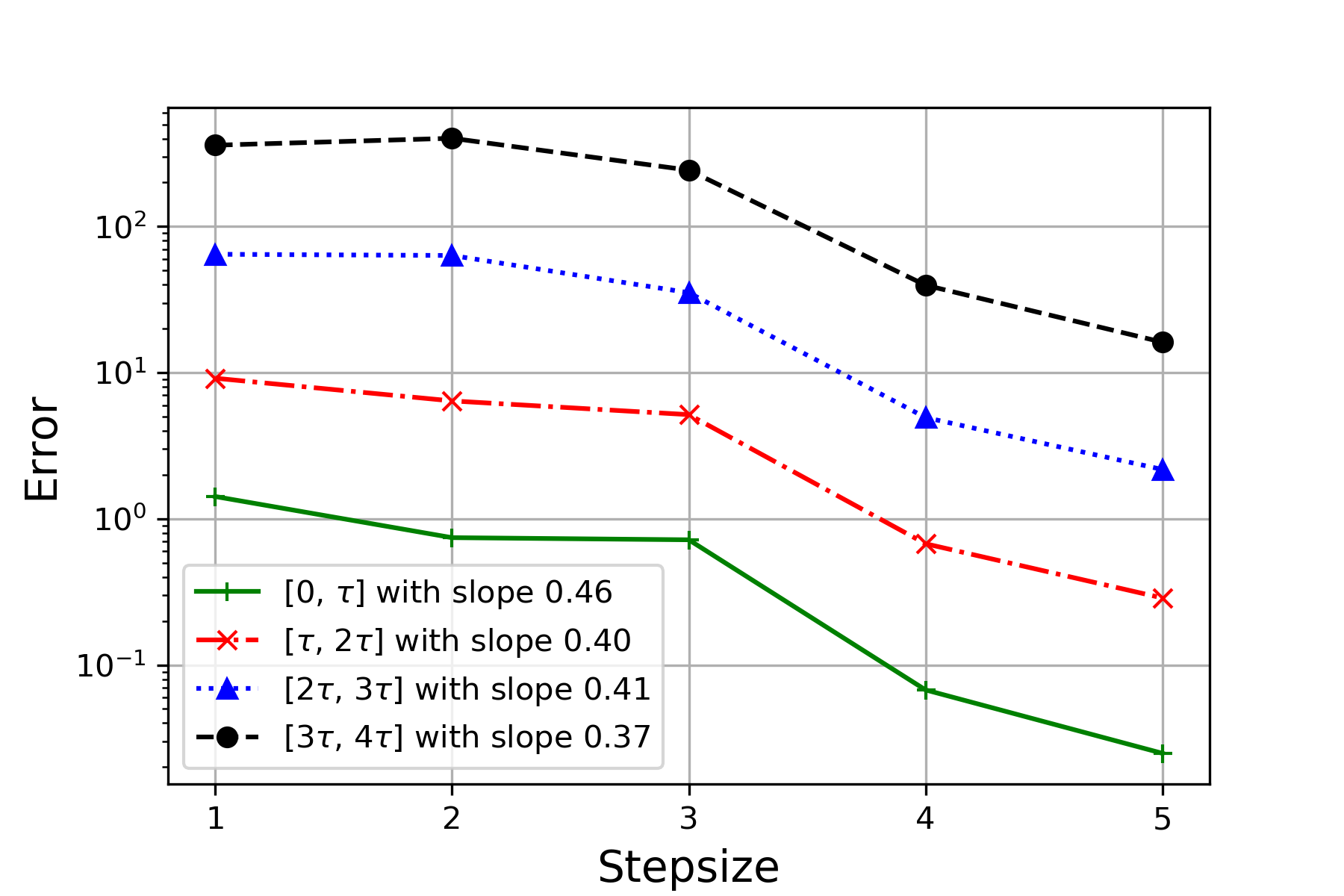}
    \caption{$\alpha=0.1$.}
    \label{fig:M10,K100,alpha0.1,gamma2.1,f6}
\end{subfigure}
\begin{subfigure}{0.32\textwidth}
    \centering
    \includegraphics[width=\textwidth]{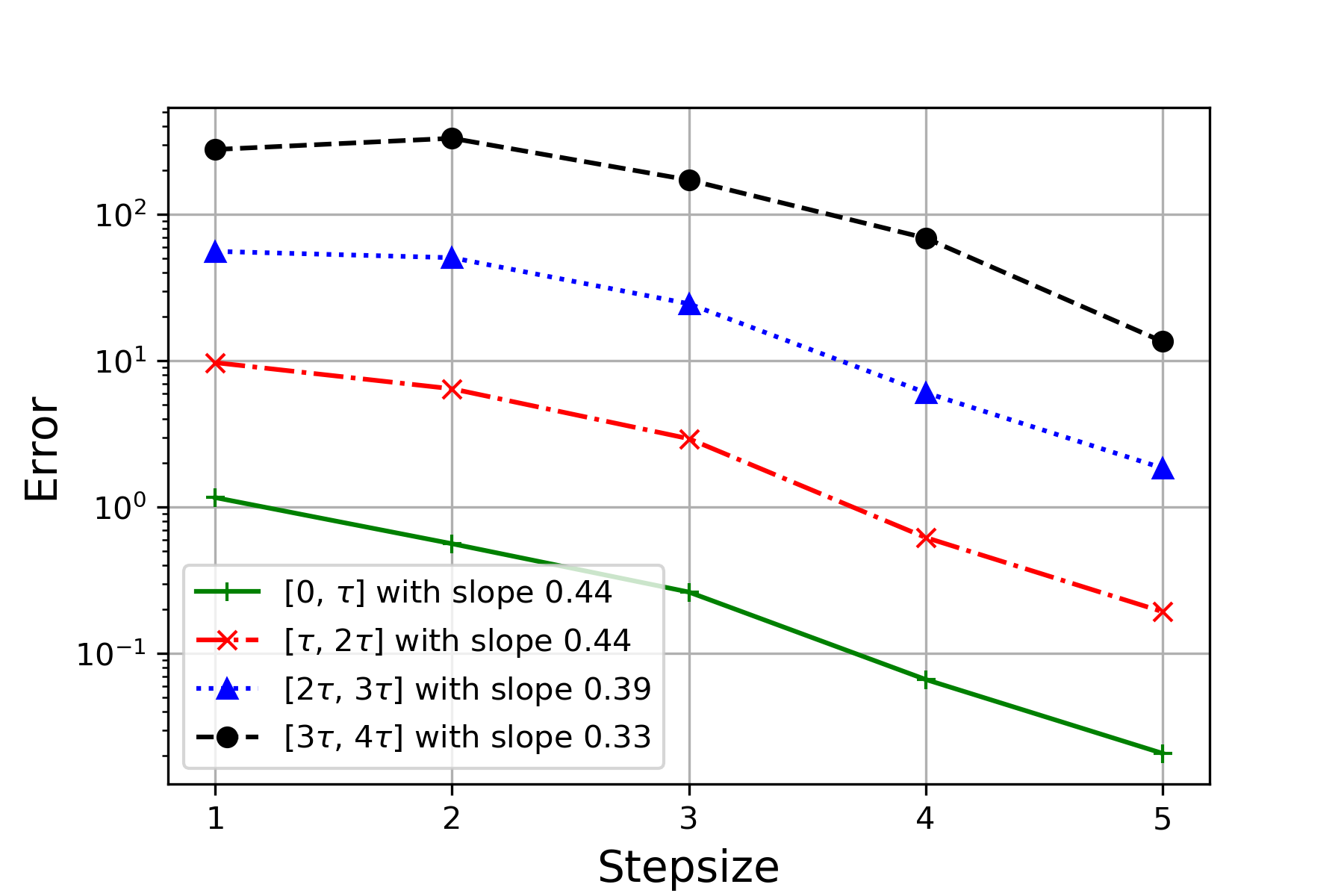}
    \caption{$\alpha=0.5$.}
    \label{fig:M10,K100,alpha0.5,gamma2.1,f6}
\end{subfigure}
\begin{subfigure}{0.32\textwidth}
    \centering
    \includegraphics[width=\textwidth]{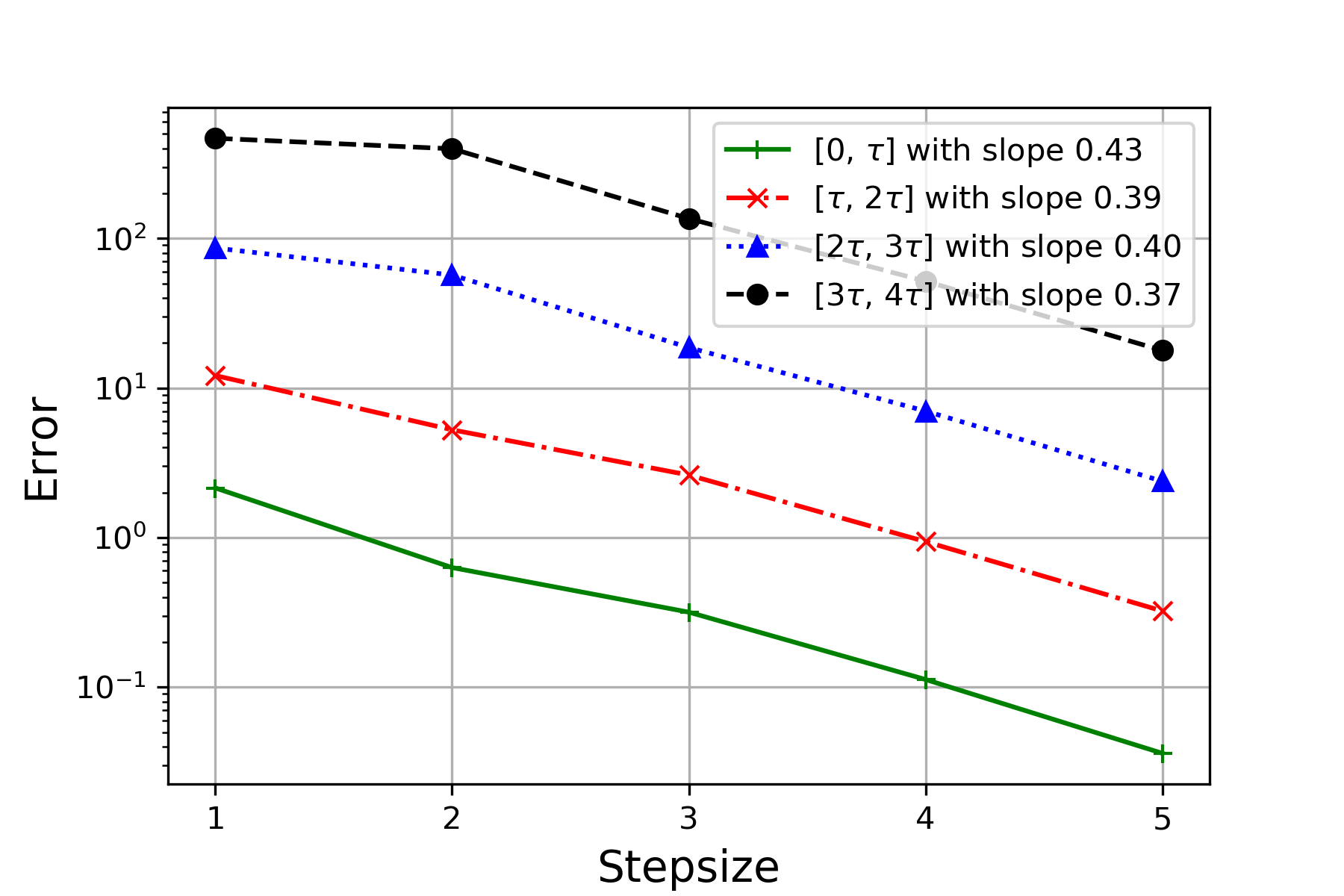}
    \caption{$\alpha=1$.}
    \label{fig:M10,K100,alpha1,gamma2.1,f6}
\end{subfigure}
\caption{Mean square errors slope for $\gamma=2.1$ and values of $\alpha=0.1,0.5,1$ in \eqref{eq:f6}, using $M=10, P=100$.}
\label{fig:mse_f6_gamma2.1}
\end{figure}
\begin{figure}
\centering
\begin{subfigure}{0.32\textwidth}
    \centering
    \includegraphics[width=\textwidth]{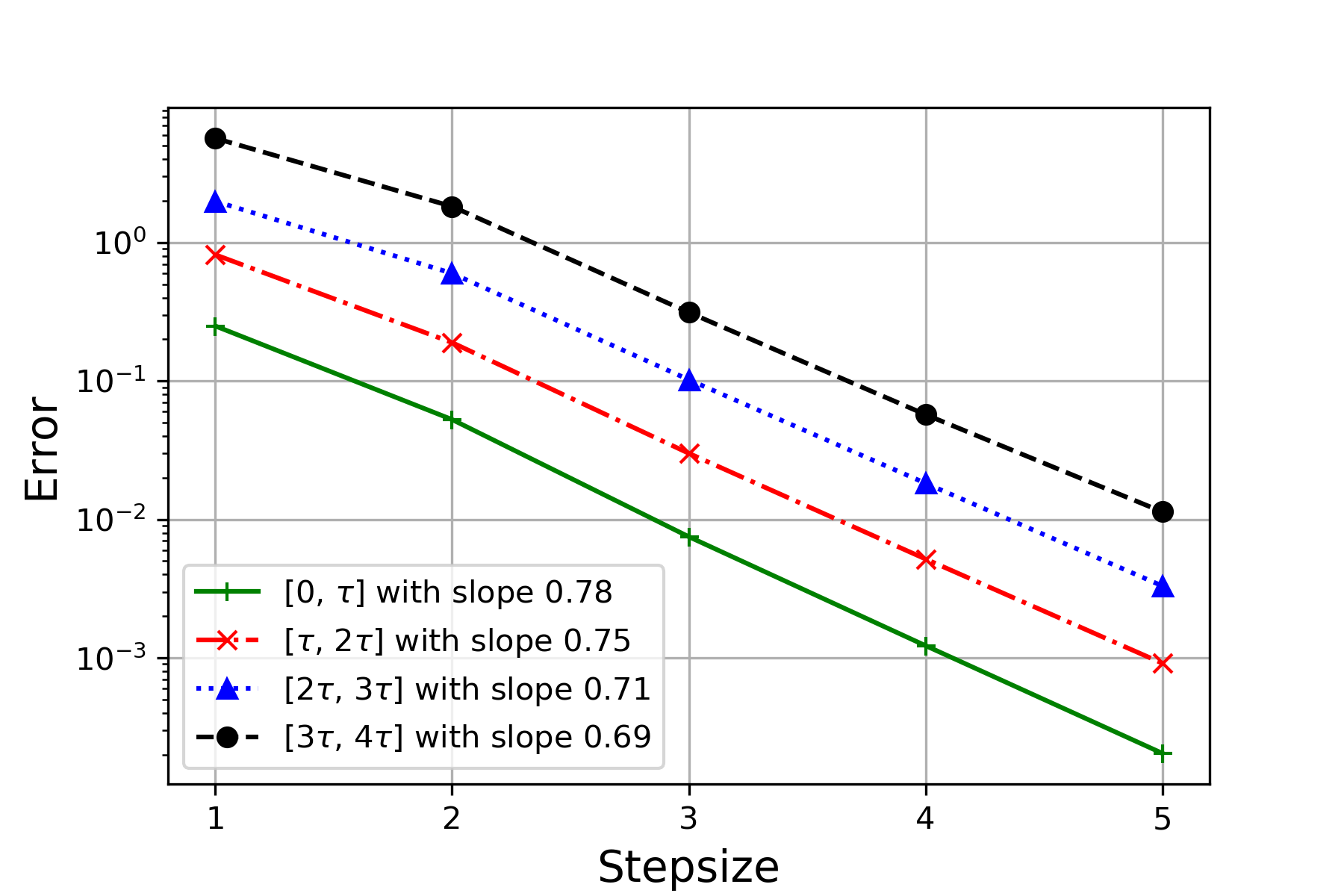}
    \caption{$\alpha=0.1$.}
    \label{fig:M10,K100,alpha0.1,gamma5,f6}
\end{subfigure}
\begin{subfigure}{0.32\textwidth}
    \centering
    \includegraphics[width=\textwidth]{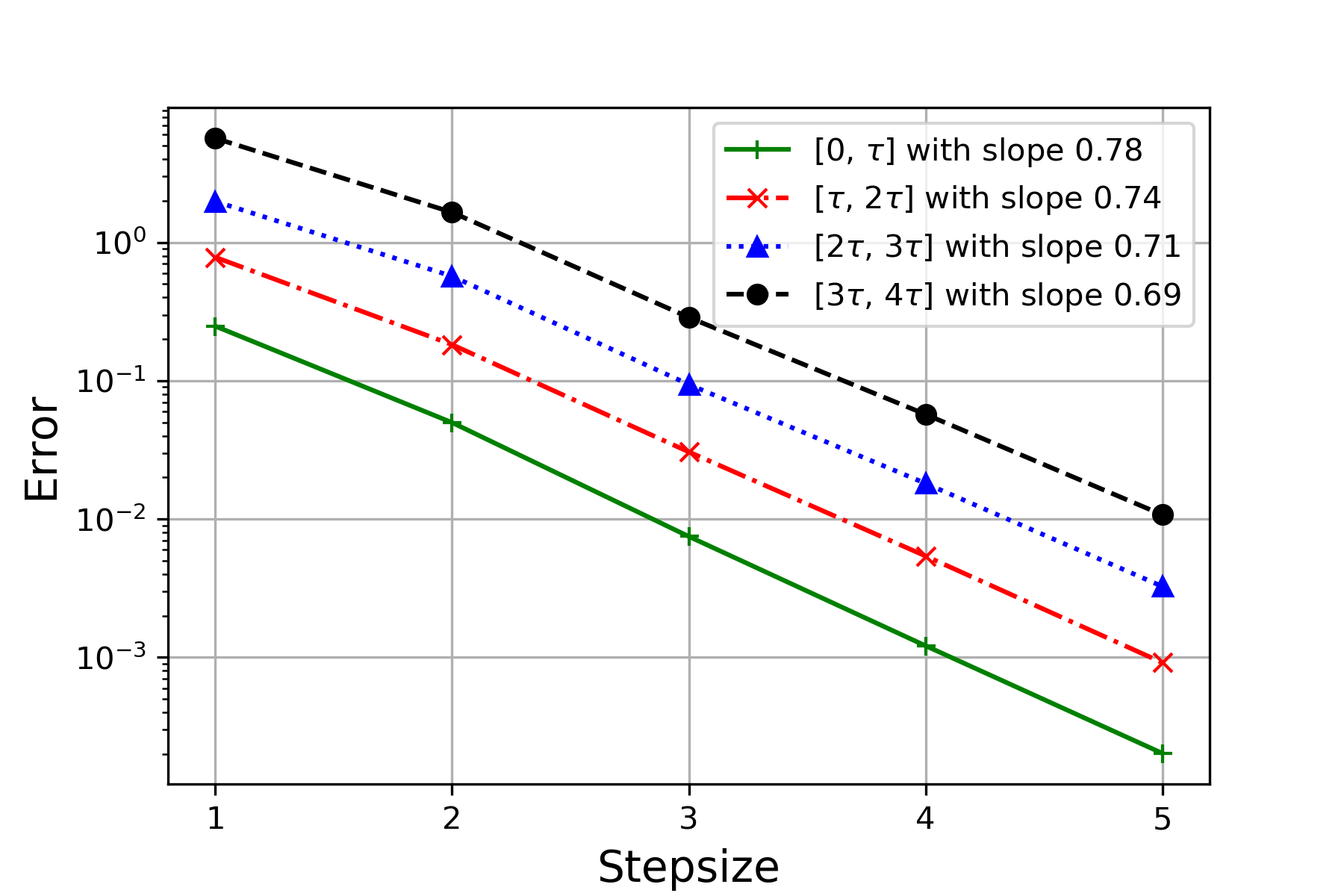}
    \caption{$\alpha=0.5$.}
    \label{fig:M10,K100,alpha0.5,gamma5,f6}
\end{subfigure}
\begin{subfigure}{0.32\textwidth}
    \centering
    \includegraphics[width=\textwidth]{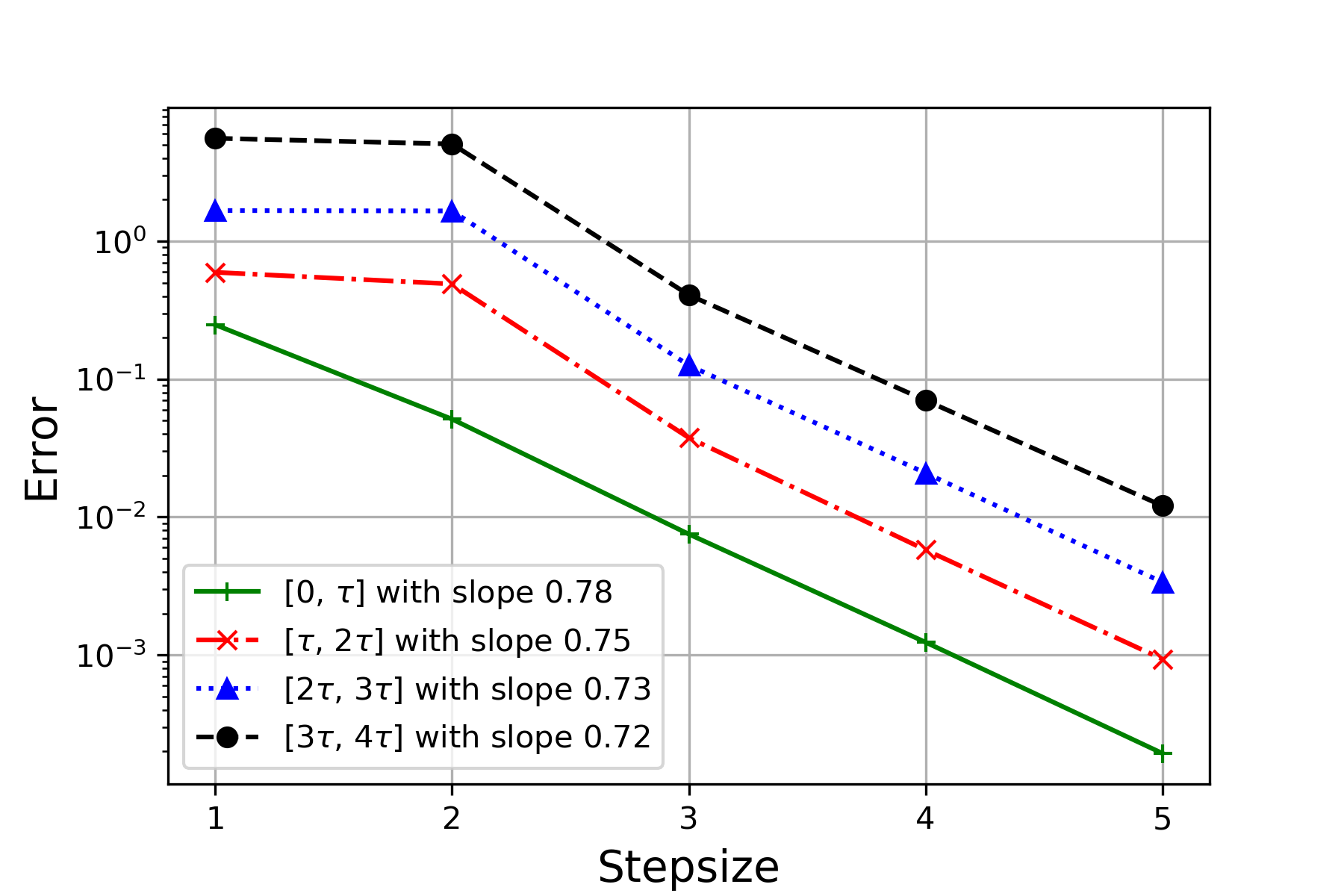}
    \caption{$\alpha=1$.}
    \label{fig:M10,K100,alpha1,gamma5,f6}
\end{subfigure}
\caption{Mean square errors slope for $\gamma=5$ and values of $\alpha=0.1,0.5,1$ in \eqref{eq:f6}, using $M=10, P=100$.}
\label{fig:mse_f6_gamma5}
\end{figure}
\end{exm}

\begin{exm}
\normalfont
In the following numerical tests we use \eqref{eq:f1} with parameters $M=10, P=100$. 
We fix the number of experiments $K=1000$ for each $N=2^l$, $l=3,\ldots,8$, and the reference solution is computed using $m=1000$; also, the horizon parameter is $n=5$.

We get the following results for $\gamma=2.1$: \\
letting $\alpha=0.1$, the negative mean square error slopes are
$0.79$, $0.77$, $0.77$ and $0.78$. See Figure \ref{fig:M10,K100,alpha0.1,gamma2.1,f2}; \\
letting $\alpha=0.5$, the negative mean square error slopes are
$0.79$, $0.78$, $0.77$ and $0.76$. See Figure \ref{fig:M10,K100,alpha0.5,gamma2.1,f2}; \\
letting $\alpha=1$,  the negative mean square error slopes are
$0.80$, $0.79$, $0.78$ and $0.77$. See Figure \ref{fig:M10,K100,alpha1,gamma2.1,f2}; \\
while, for $\gamma=5$: \\
letting $\alpha=0.1$, the negative mean square error slopes are
$0.80$, $0.78$, $0.77$ and $0.76$. See Figure \ref{fig:M10,K100,alpha0.1,gamma5,f1}; \\
letting $\alpha=0.5$, the negative mean square error slopes are
$0.79$, $0.79$, $0.78$ and $0.77$. See Figure \ref{fig:M10,K100,alpha0.5,gamma5,f1}; \\
letting $\alpha=1$, the negative mean square error slopes are
$0.80$, $0.80$, $0.79$ and $0.78$. See Figure \ref{fig:M10,K100,alpha1,gamma5,f1}.
% \begin{figure}
% \centering
% \begin{subfigure}{0.3\textwidth}
%     \centering
%     \includegraphics[width=\textwidth]{Figures/m10k100alpha01gamma21_ex2.png}
%     \caption{Log-plot of the Mean-Square Error \eqref{error_main_thm_alpha1} relative to \eqref{eq:f1} and $\alpha=0.1$.}
%     \label{fig:M10,K100,alpha0.1,gamma2.1,f1}
% \end{subfigure}
% \hfill
% \begin{subfigure}{0.3\textwidth}
%     \centering
%     \includegraphics[width=\textwidth]{Figures/m10k100alpha05gamma21_ex2.png}
%     \caption{Log-plot of the Mean-Square Error \eqref{error_main_thm_alpha1} relative to \eqref{eq:f1} and $\alpha=0.5$.}
%     \label{fig:M10,K100,alpha0.5,gamma2.1,f1}
% \end{subfigure}
% \hfill
% \begin{subfigure}{0.3\textwidth}
%     \centering
%     \includegraphics[width=\textwidth]{Figures/m10k100alpha1gamma21_ex2.png}
%     \caption{Log-plot of the Mean-Square Error \eqref{error_main_thm_alpha1} relative to \eqref{eq:f1} and $\alpha=1$.}
%     \label{fig:M10,K100,alpha1,gamma2.1,f1}
% \end{subfigure}
% \caption{Mean square errors slope for $\gamma=2.1$ and values of $\alpha=0.1,0.5,1$ in \eqref{eq:f1}, using $M=10, P=100$.}
% \label{fig:mse_f1_gamma2.1}
% \end{figure}
\begin{figure}
\centering
\begin{subfigure}{0.32\textwidth}
    \centering
    \includegraphics[width=\textwidth]{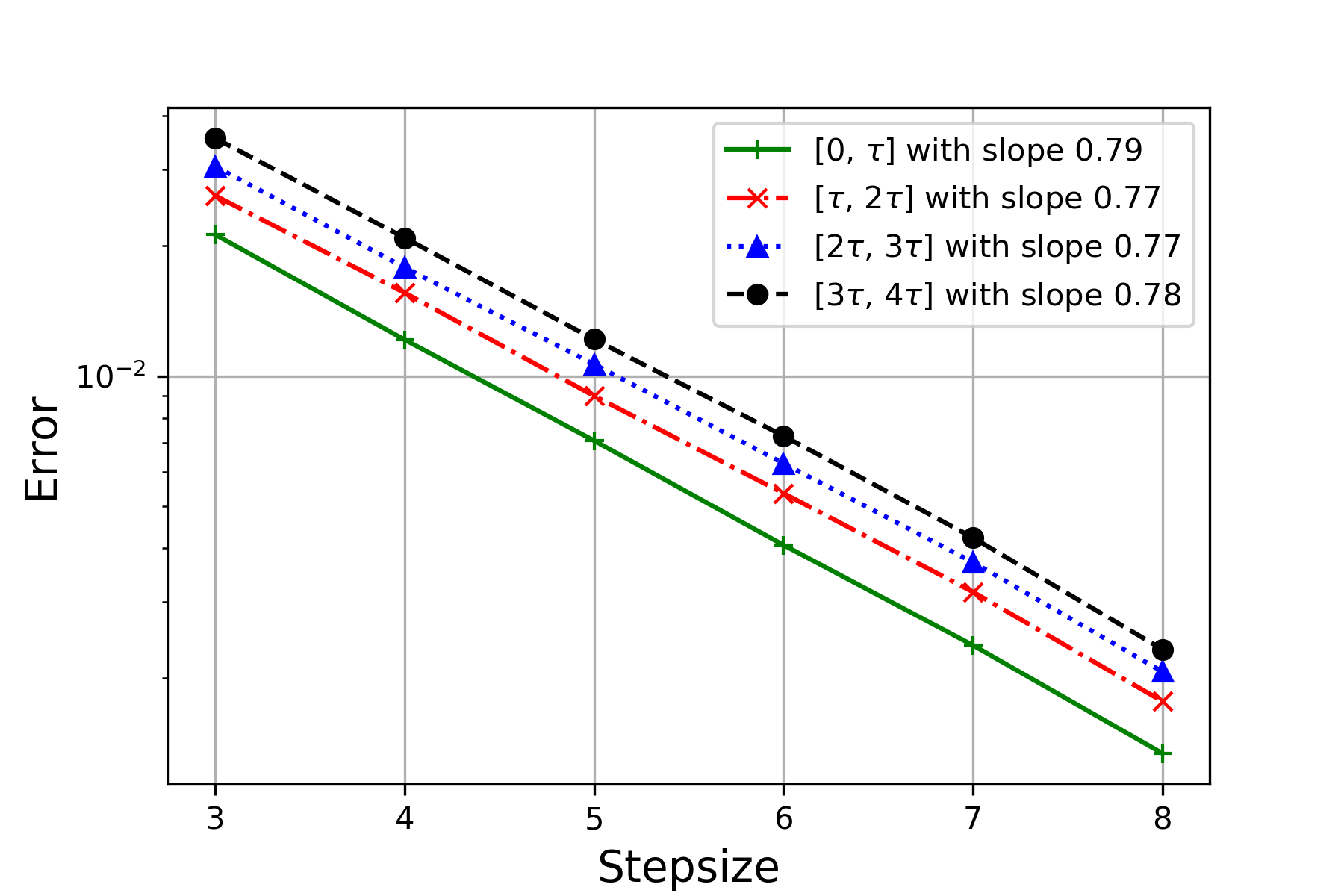}
    \caption{$\alpha=0.1$.}
    \label{fig:M10,K100,alpha0.1,gamma2.1,f2}
\end{subfigure}
\begin{subfigure}{0.32\textwidth}
    \centering
    \includegraphics[width=\textwidth]{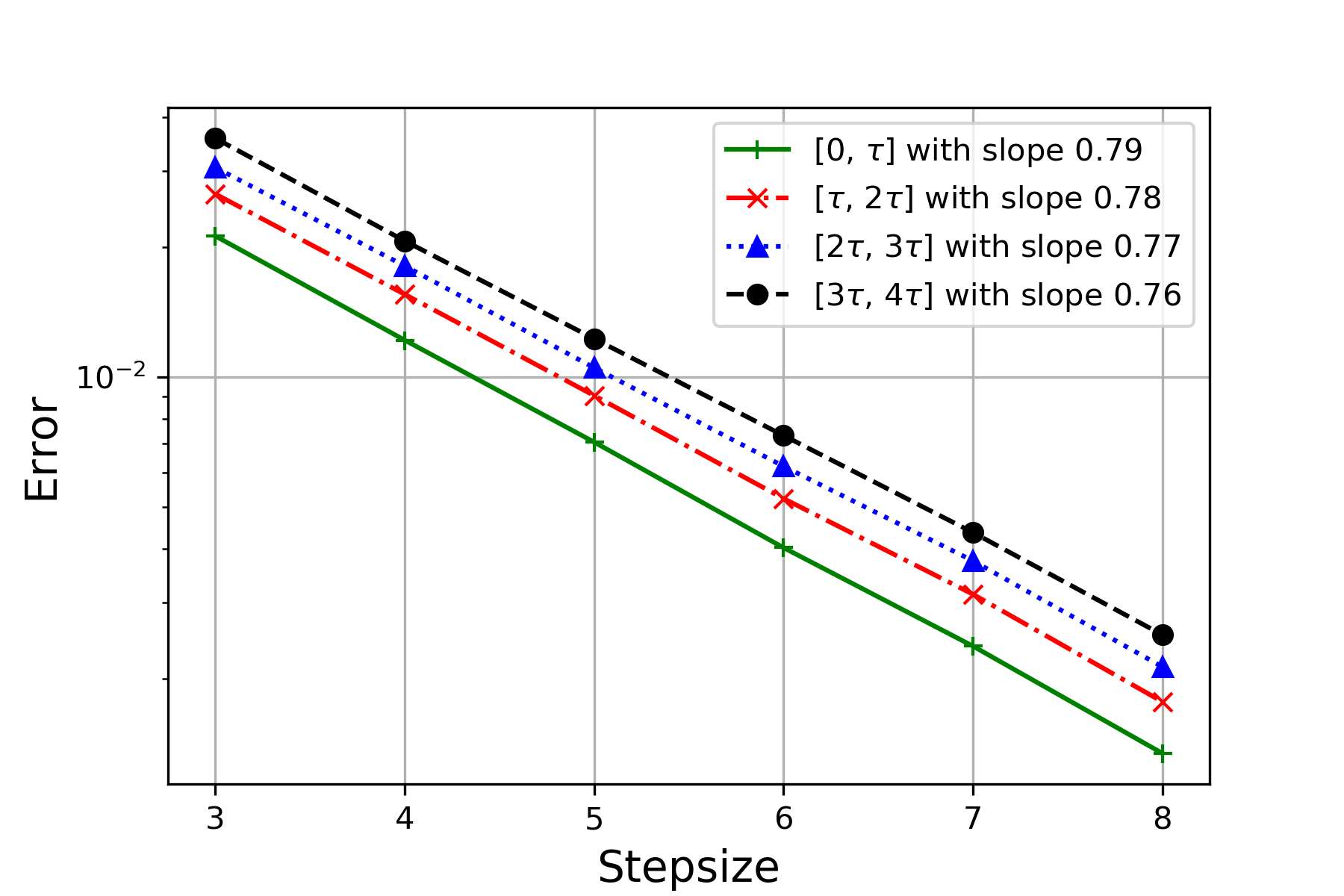}
    \caption{$\alpha=0.5$.}
    \label{fig:M10,K100,alpha0.5,gamma2.1,f2}
\end{subfigure}
\begin{subfigure}{0.32\textwidth}
    \centering
    \includegraphics[width=\textwidth]{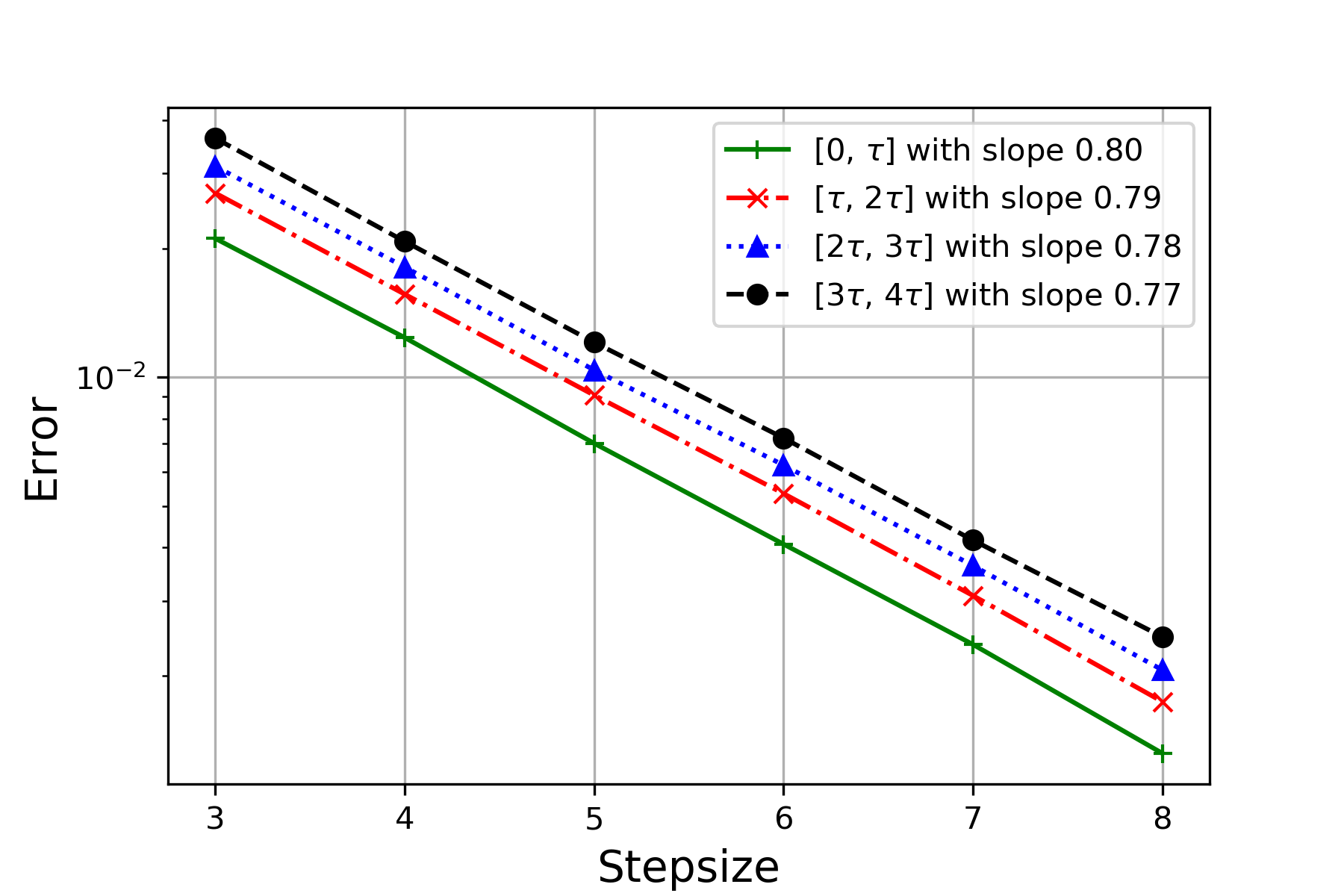}
    \caption{$\alpha=1$.}
    \label{fig:M10,K100,alpha1,gamma2.1,f2}
\end{subfigure}
\caption{Mean square errors slope for $\gamma=2.1$ and values of $\alpha=0.1,0.5,1$ in \eqref{eq:f1}, using $M=10, P=100$.}
\label{fig:mse_f1_gamma2.1}
\end{figure}
% \begin{figure}
% \centering
% \begin{subfigure}{0.3\textwidth}
%     \centering
%     \includegraphics[width=\textwidth]{Figures/m10k100alpha01gamma21_ex2.png}
%     \caption{Log-plot of the Mean-Square Error \eqref{error_main_thm_alpha1} relative to \eqref{eq:f1} and $\alpha=0.1$.}
%     \label{fig:M10,K100,alpha0.1,gamma2.1,f1}
% \end{subfigure}
% \hfill
% \begin{subfigure}{0.3\textwidth}
%     \centering
%     \includegraphics[width=\textwidth]{Figures/m10k100alpha05gamma21_ex2.png}
%     \caption{Log-plot of the Mean-Square Error \eqref{error_main_thm_alpha1} relative to \eqref{eq:f1} and $\alpha=0.5$.}
%     \label{fig:M10,K100,alpha0.5,gamma2.1,f1}
% \end{subfigure}
% \hfill
% \begin{subfigure}{0.3\textwidth}
%     \centering
%     \includegraphics[width=\textwidth]{Figures/m10k100alpha1gamma21_ex2.png}
%     \caption{Log-plot of the Mean-Square Error \eqref{error_main_thm_alpha1} relative to \eqref{eq:f1} and $\alpha=1$.}
%     \label{fig:M10,K100,alpha1,gamma2.1,f1}
% \end{subfigure}
% \caption{Mean square errors slope for $\gamma=2.1$ and values of $\alpha=0.1,0.5,1$ in \eqref{eq:f1}, using $M=10, P=100$.}
% \label{fig:mse_f1_gamma2.1}
% \end{figure}
\begin{figure}
\centering
\begin{subfigure}{0.3\textwidth}
    \centering
    \includegraphics[width=\textwidth]{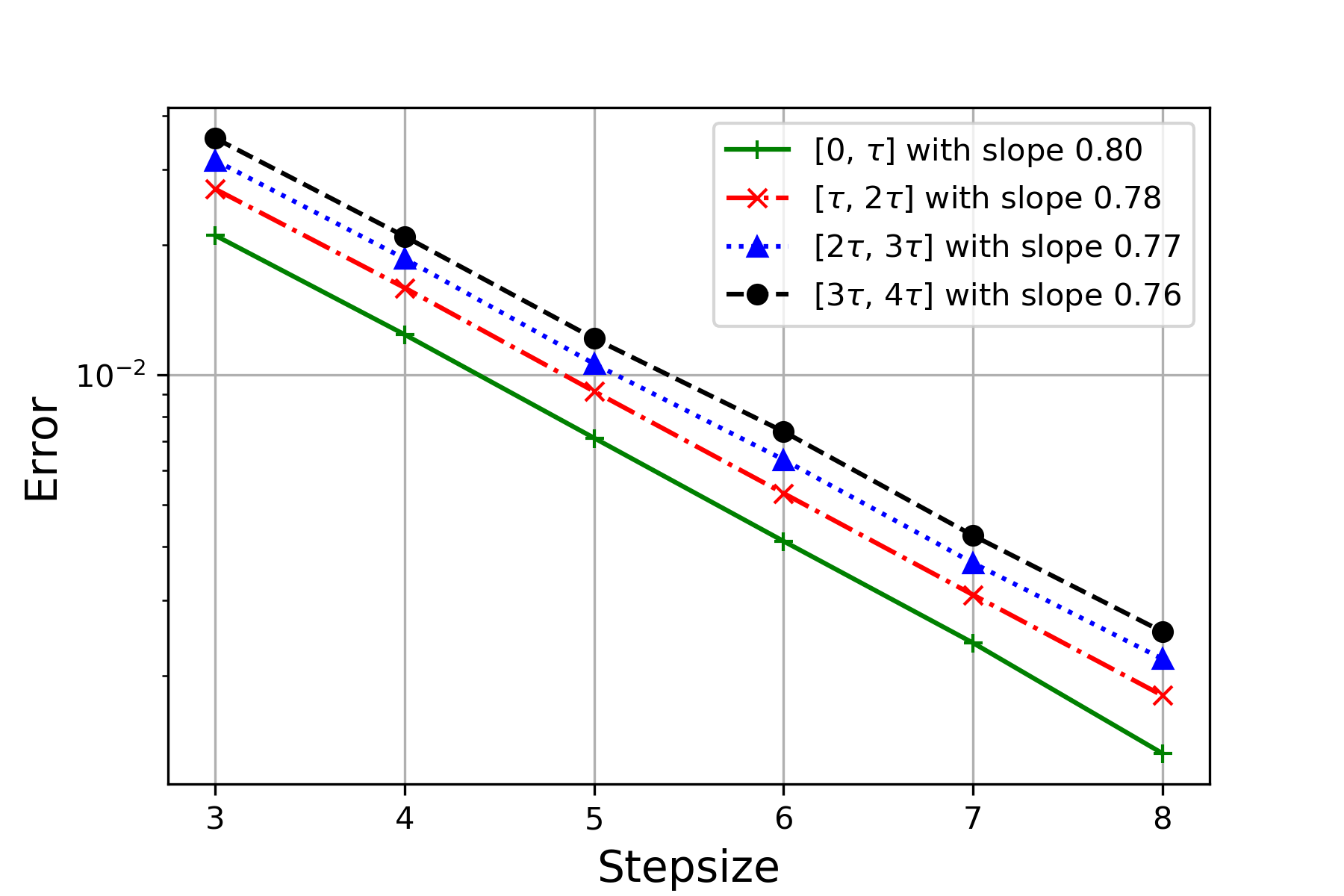}
    \caption{$\alpha=0.1$.}
    \label{fig:M10,K100,alpha0.1,gamma5,f1}
\end{subfigure}
\hfill
\begin{subfigure}{0.3\textwidth}
    \centering
    \includegraphics[width=\textwidth]{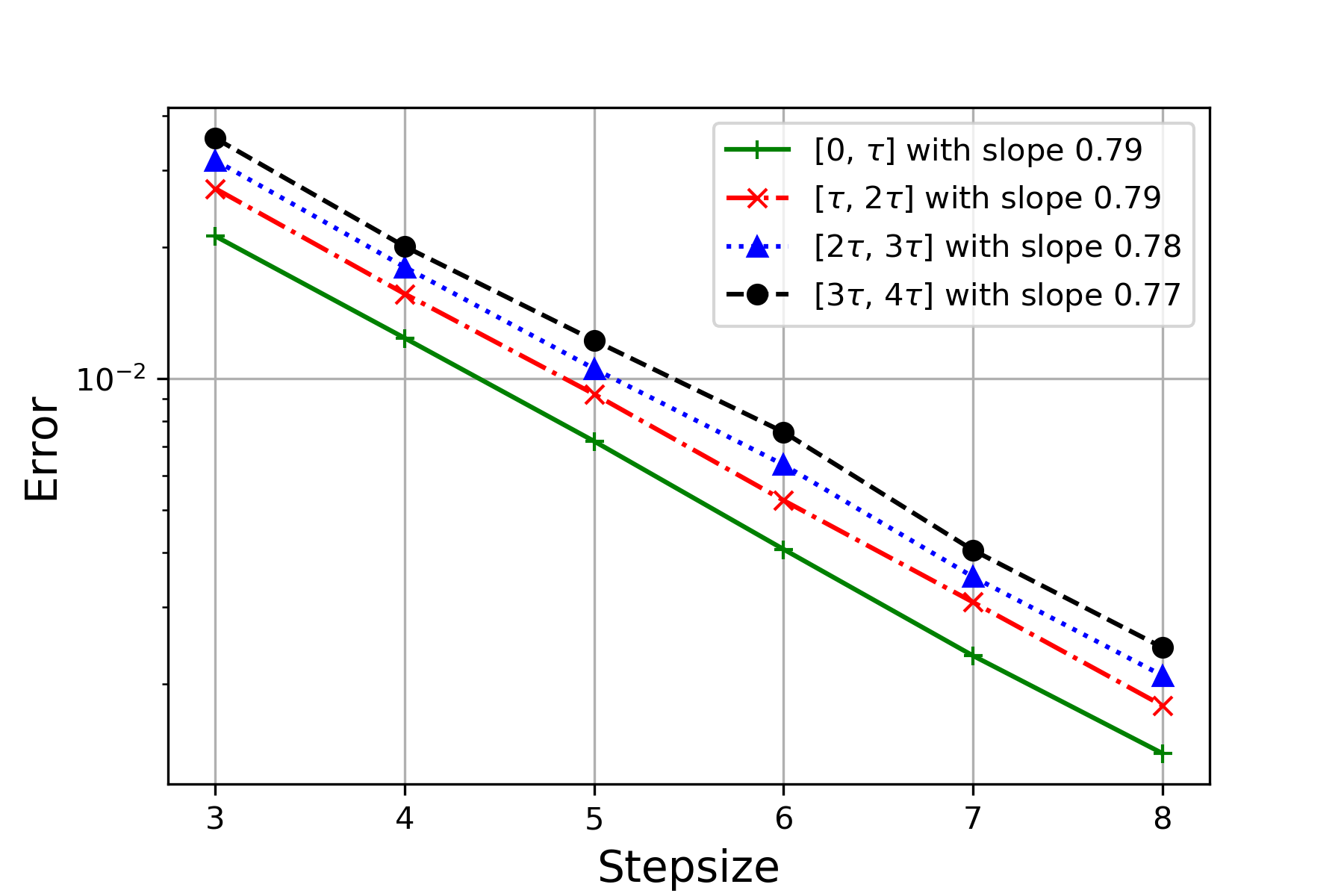}
    \caption{$\alpha=0.5$.}
    \label{fig:M10,K100,alpha0.5,gamma5,f1}
\end{subfigure}
\hfill
\begin{subfigure}{0.3\textwidth}
    \centering
    \includegraphics[width=\textwidth]{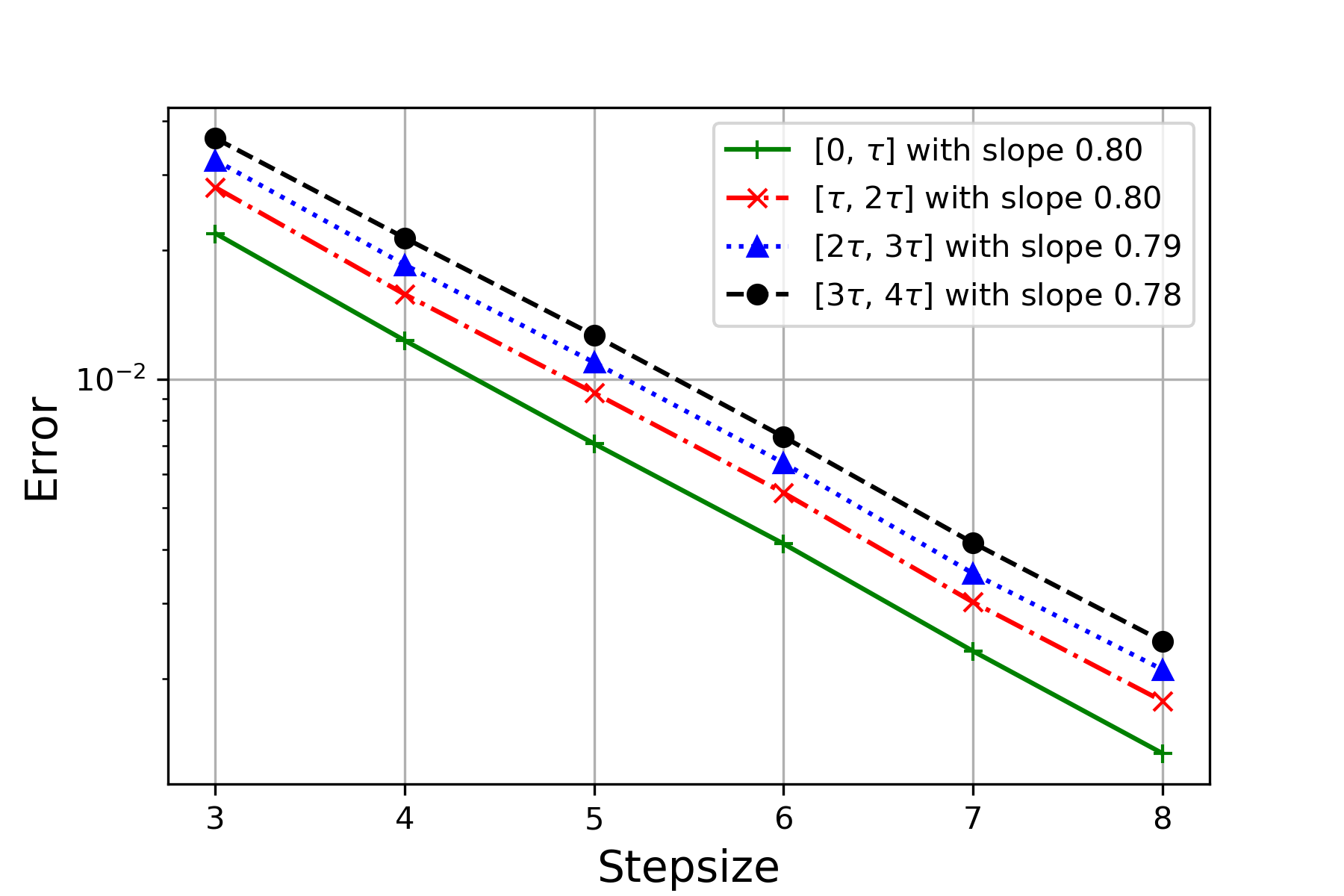}
    \caption{$\alpha=1$.}
    \label{fig:M10,K100,alpha1,gamma5,f1}
\end{subfigure}
\caption{Mean square errors slope for $\gamma=5$ and values of $\alpha=0.1,0.5,1$ in \eqref{eq:f1}, using $M=10, P=100$.}
\label{fig:mse_f1_gamma5}
\end{figure}
% \begin{figure}
% \centering
% \begin{subfigure}{0.3\textwidth}
%     \centering
%     \includegraphics[width=\textwidth]{Figures/m10k100alpha01gamma5_ex2.png}
%     \caption{Log-plot of the Mean-Square Error \eqref{error_main_thm_alpha1} relative to \eqref{eq:f1} and $\alpha=0.1$.}
%     \label{fig:M10,K100,alpha0.1,gamma5,f1}
% \end{subfigure}
% \hfill
% \begin{subfigure}{0.3\textwidth}
%     \centering
%     \includegraphics[width=\textwidth]{Figures/m10k100alpha05gamma5_ex2.png}
%     \caption{Log-plot of the Mean-Square Error \eqref{error_main_thm_alpha1} relative to \eqref{eq:f1} and $\alpha=0.5$.}
%     \label{fig:M10,K100,alpha0.5,gamma5,f1}
% \end{subfigure}
% \hfill
% \begin{subfigure}{0.3\textwidth}
%     \centering
%     \includegraphics[width=\textwidth]{Figures/m10k100alpha1gamma5_ex2.png}
%     \caption{Log-plot of the Mean-Square Error \eqref{error_main_thm_alpha1} relative to \eqref{eq:f1} and $\alpha=1$.}
%     \label{fig:M10,K100,alpha1,gamma5,f1}
% \end{subfigure}
% \caption{Mean square errors slope for $\gamma=5$ and values of $\alpha=0.1,0.5,1$ in \eqref{eq:f1}, using $M=10, P=100$.}
% \label{fig:mse_f1_gamma5}
% \end{figure}
\end{exm}

\begin{exm}\label{ex:kainhofer}
\normalfont
From \cite{kainhofer} we consider the following DDE
\begin{equation}\label{eq:kainhofer}
\begin{cases}
x'(t)=3x(t-\tau)\sin(\lambda t), & t\geq0, \\
x(t)=x_0, & t\leq0,
\end{cases}
\end{equation}
with $\tau=1$, $\lambda=2^\nu$ and $1\leq\nu\leq16$.
Thus, with our formalism, we have to consider the function
\[
f(t,x,z)\udef 3z\sin(\lambda t),
\]
which satisfies assumptions \ref{ass:A1}, \ref{ass:A2}, \ref{ass:A3'}.  In this case the exact solution can be computed in closed form. Namely, by using the fact that for $j\geq 0$ and $t\in [j\tau,(j+1)\tau]$ it holds
\begin{equation}
    \phi_j(t)=\phi_{j-1}(j\tau)+3\int\limits_{j\tau}^t\phi_{j-1}(u-\tau)\sin(\lambda u)\mathrm{d}u.
\end{equation}
In this experiment, we consider $t\in [0,2\tau]$, thus the true solutions over $[0,\tau]$ and $[\tau,2\tau]$ are therefore
\begin{itemize}
    \item for $t\in [0,\tau]$
    \begin{equation}\label{eqn:true_sol1}
        \phi_0(t)=x_0+3x_0\frac{1-\cos(\lambda t)}{\lambda},
    \end{equation}
    \item for $t\in [\tau,2\tau]$
    \begin{align}\label{eqn:true_sol2}
    \begin{split}
     \phi_1(t)=&\phi_0(\tau)-\big(\frac{9x_0}{\lambda^2}+\frac{3x_0}{\lambda}\big)(\cos(\lambda t)-\cos(\lambda \tau))+\frac{9x_0}{2\lambda}(t-\tau) \sin(-\lambda \tau)\\
  &+\frac{9x_0}{4\lambda^2}(\cos(2\lambda t-\lambda\tau)-\cos(\lambda \tau)).    
    \end{split}
    \end{align}
%     \begin{eqnarray}
%         % &&\phi_1(t)=\phi_0(\tau)-\frac{3x_0}{4\lambda^2}\Bigl[4(3+\lambda)\cos(\lambda t)-(9+4\lambda)\cos(\lambda\tau)\notag\\
%         % &&-3\Bigl(\cos(\lambda(-2t+\tau))+2\lambda(-t+\tau)\sin(\lambda\tau)\Bigr)\Bigr],
%   &&\phi_1(t)=\phi_0(\tau)-\frac{9x_0}{\lambda^2}(\cos(\lambda t)-\cos(\lambda \tau))+\frac{9x_0}{2\lambda}(t-\tau) \sin(-\lambda t)\\
%   &&+\frac{9x_0}{4\lambda^2}(\cos(2\lambda t-\lambda\tau)-\cos(\lambda \tau)).
%     \end{eqnarray}
\end{itemize}
Exact solutions for $\lambda=2^\nu$, $\nu=1,\ldots,9$ are depicted in Figure \ref{fig:kainhoferSol}.
\begin{figure}
    \centering
    \includegraphics[width=0.6\textwidth]{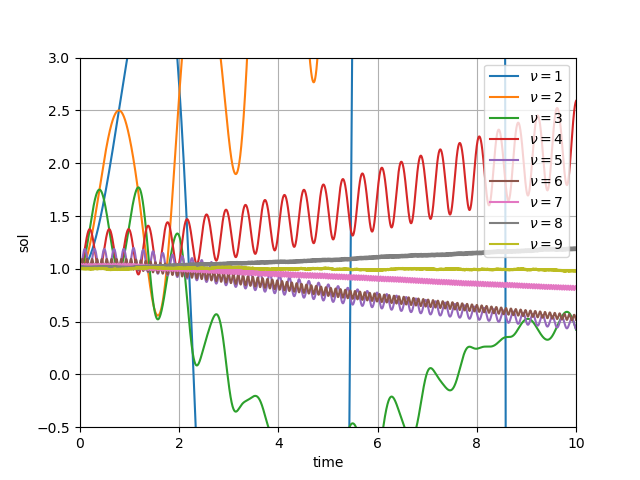}
    \caption{Solutions to problem \eqref{eq:kainhofer}.}
    \label{fig:kainhoferSol}
\end{figure}
We fix the number of experiments $K=1000$ for each $N=2^l$, $l=5,\ldots,10$, choose $x_0=1$.
% The true solution is computed through Eqn. \eqref{eqn:true_sol1} and \eqref{eqn:true_sol2}. 
With $\nu=8$ the errors from both randomized Euler and Euler schemes are depicted in Figure \ref{fig:kainhofer_nu1}, where the slopes are both beyond $1.00$ over $[0,\tau]$, and Euler method has a slower convergence over $[\tau,2\tau]$. %$-0.54925095$. 
% Using $x(t)=-1$ for $t\leq0$ as initial condition, the slope is $0.99$ and evidence of convergence is shown in Figure \ref{fig:kainhofer_nu1_ICm1}. 

\begin{figure}
\begin{subfigure}{0.45\textwidth}
    \centering
    \includegraphics[width=\textwidth]{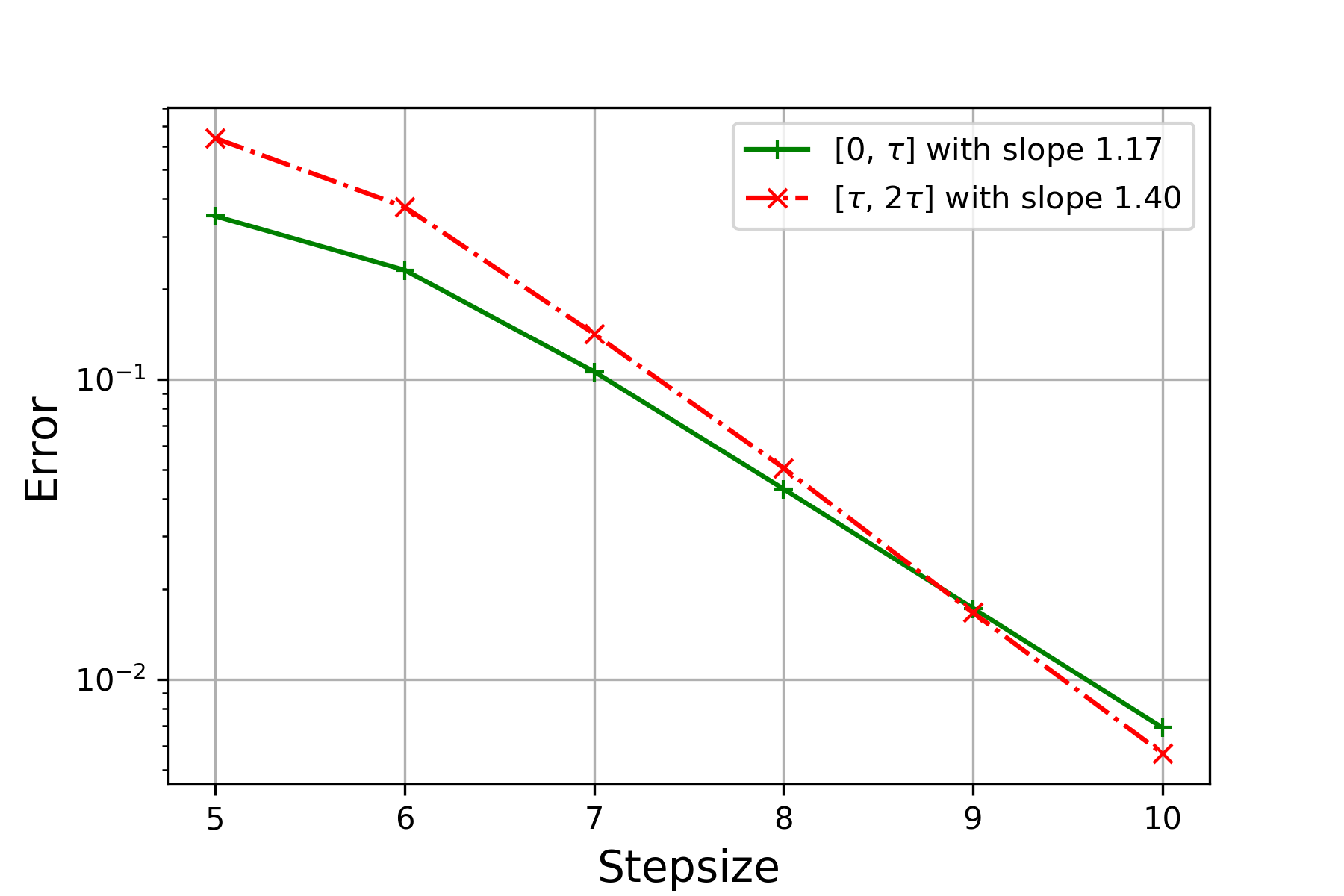}
    \caption{Randomized Euler method}
    \label{fig:kainhofer_nu1_IC1}
\end{subfigure}
\hfill
\begin{subfigure}{0.45\textwidth}
    \centering
    \includegraphics[width=\textwidth]{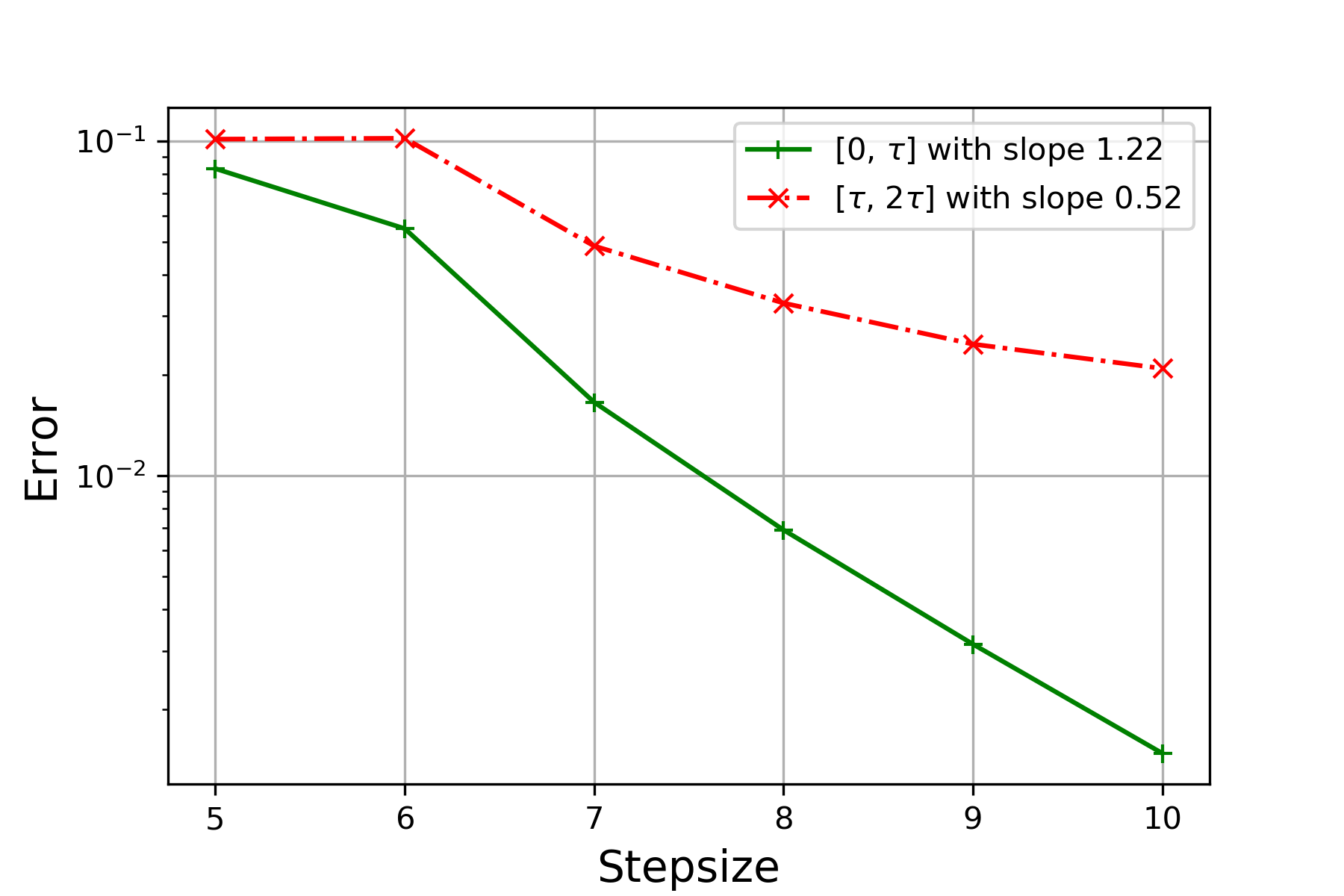}
    \caption{Euler method}
    \label{fig:kainhofer_nu1_ICm1}
\end{subfigure}
\caption{Error comparison between randomized Euler method and
Euler method for Example \ref{ex:kainhofer} with $\nu=8$ over $[0, \tau]$ and $[ \tau, 2\tau]$.}
\label{fig:kainhofer_nu1}
\end{figure}

\end{exm}

\begin{exm}\label{ex:comparison}
\normalfont
Following the example in \cite[Section 7]{kruse2017randomized} we consider the following DDE
\begin{equation}\label{eq:comparison}
\begin{cases}
x'(t)=\sin{(\lambda_1 t)}+x(t)+|x(t-1)|^\alpha, & t\geq0, \\
x(t)=1, & t\leq0.
\end{cases}
\end{equation}
Thus, with our formalism, we have to consider the function
\[
f(t,x,z)\udef \sin{(\lambda_1 t)}+x+|z|^\alpha,
\]
which satisfies assumptions \ref{ass:A1}, \ref{ass:A2}, \ref{ass:A3'}.  \\ 
In the experiment, we set  $\lambda_1=2^9 \pi$ and $\alpha=0.2$.
We compare the numerical solution of \eqref{eq:comparison} by the randomized Euler scheme \eqref{expl_euler_1}-\eqref{expl_euler_11} and its classical counter-part over $[j\tau,(j+1)\tau]$, for $j\in\{0,1,2\}$ respectively, where $\tau=1$ in this example. We approximate the error by a Monte Carlo simulation
with 1000 independent samples. Hereby, the reference solution is obtained using
the randomized Euler scheme with a finer step size of $h_{\text{ref}}$ = $2^{-16}$.

In Figure \ref{fig:comparison_nu1}, we plot the errors against the underlying step
size, i.e., the number $i$ on the x-axis indicates the corresponding simulation is based
on the step size $h=2^{-i}$ . The finest step size here is $2^{-10}$. The two sets of error
data are fitted with a linear function via linear regression respectively, where the
slope indicates the average order of convergence. It is noted that the
classical Euler scheme does not begin to converge until $i = 9$. The reason for this is, that for any coarser (equidistant) step size larger than $2^{-9}$ the classical Euler scheme cannot distinguish the term $|\sin(\lambda_1 t)|$ from the zero
function. In contrast, the randomized Euler method shows better results already for much coarser step sizes. Note that the experimental order of convergence is decreasing with $j$ for randomized Euler as expected (see the result for $\alpha<1$ in Theorem \ref{rate_of_conv_expl_Eul}).
\begin{figure}
\begin{subfigure}{0.475\textwidth}
    \centering
    \includegraphics[width=\textwidth]{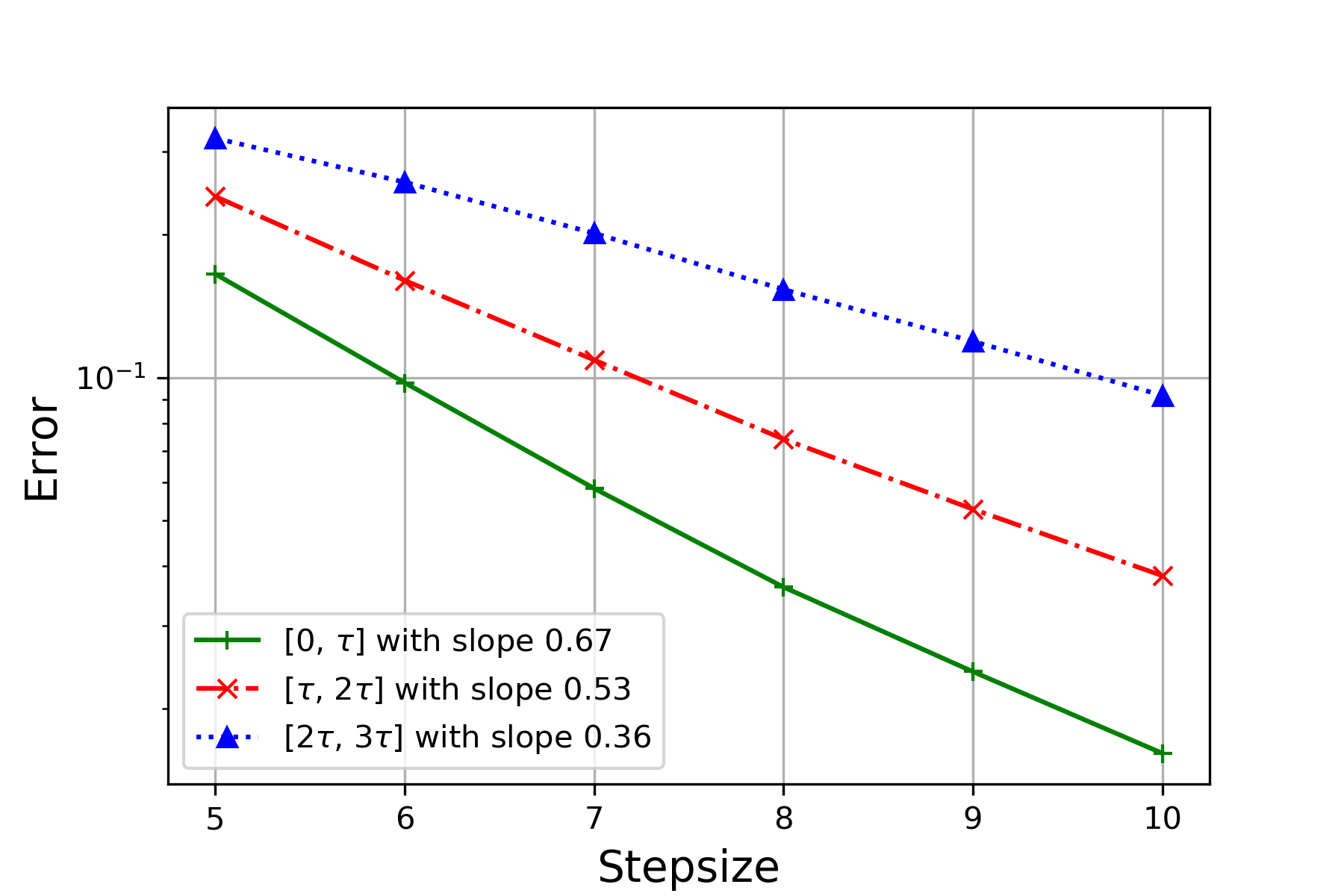}
    \caption{Randomized Euler method}
    \label{fig:comparison_nu1_IC1}
\end{subfigure}
\begin{subfigure}{0.475\textwidth}
    \centering
    \includegraphics[width=\textwidth]{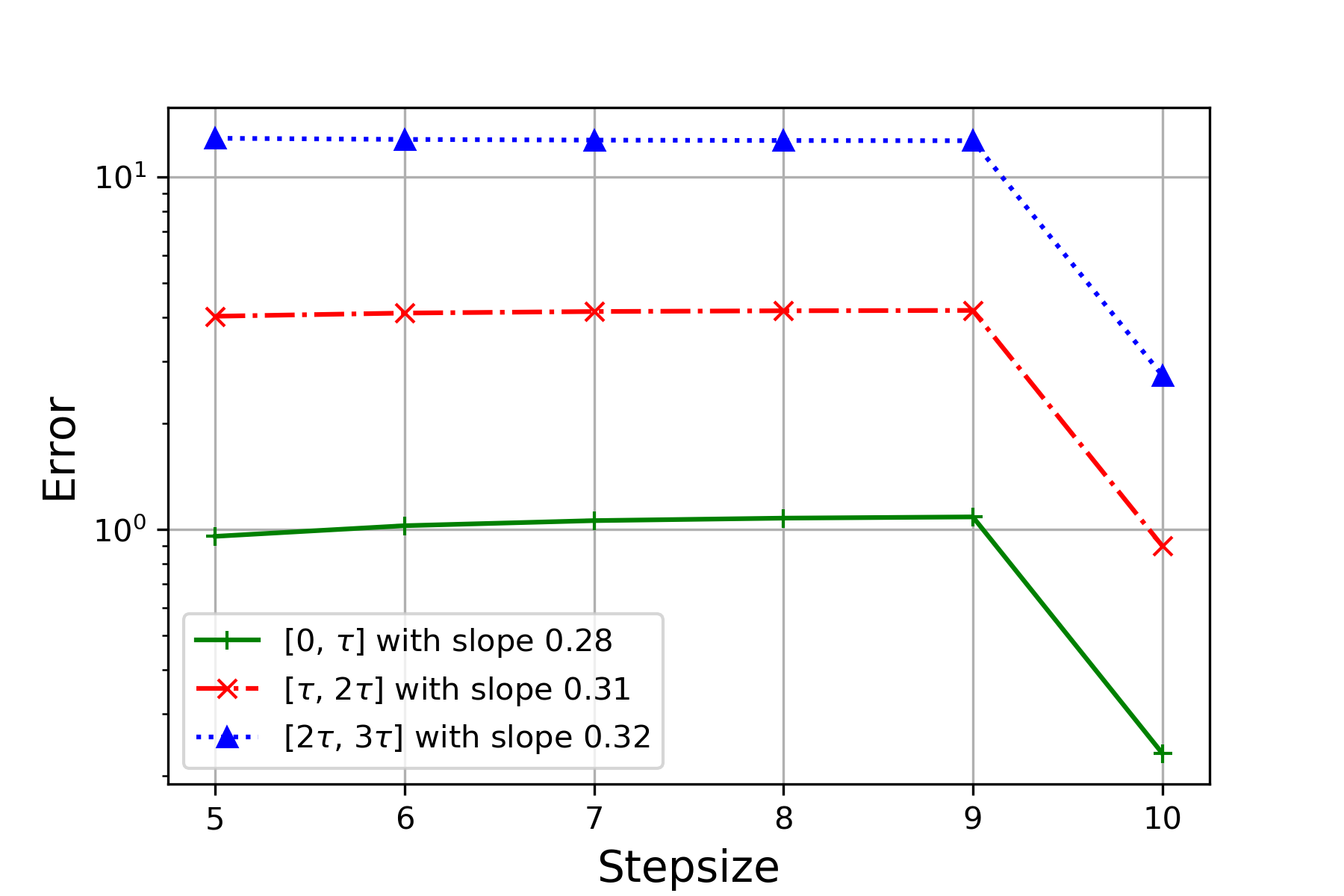}
    \caption{Euler method}
    \label{fig:comparison_nu1_ICm1}
\end{subfigure}
\caption{Error comparison between randomized Euler method and Euler method for Example \ref{ex:comparison} over $[0,\tau]$, $[\tau,2\tau]$ and $[2\tau,3\tau]$.}
\label{fig:comparison_nu1}
\end{figure}
\end{exm}

\begin{exm}\label{ex:BM}
\normalfont
Finally we consider the inspiring example in \eqref{eq:DiscDDE21} over $[0,2\tau]$ with $U_0=1$, where $Z(t)$ is given as one realization of Wiener process\footnote{It is well known that Wiener process is $0.5-\epsilon$-H\"older continuous for arbitrarily small $\epsilon>0$.}, and $a(x,z):=x+z$.  
Thus, with our formalism, we have to consider the function
\[
f(t,x,z)\udef a(x+Z(t),z+Z(t-\tau))=x+Z(t)+z+Z(t-\tau),
\]
which clearly satisfies assumptions \ref{ass:A1}, \ref{ass:A2}, \ref{ass:A3'}.  \\ 
In the experiment, we compare the numerical solution of \eqref{eq:comparison} by the randomized Euler scheme \eqref{expl_euler_1}-\eqref{expl_euler_11} and its classical counter-part over $[j\tau,(j+1)\tau]$, for $j\in\{0,1\}$ respectively, where $\tau=2$ in this example. We approximate the error by a Monte Carlo simulation
with 1000 independent samples. Hereby, the reference solution is obtained using
the Euler scheme with a finer step size of $h_{\text{ref}}$ = $2^{-16}$. Note that the realization of Wiener process is generated on the time grid with stepsize $h_{\text{ref}}$ over $[0,3\tau]$, while $Z(t)=0$ for $t<0$. The whole path evaluation $Z(t)$ for $t\in [-\tau,3\tau]$ is obtained through piecewise linear interpolation. 

In Figure \ref{fig:comparison_BM}, we plot the errors against the underlying step
size, i.e., the number $i$ on the x-axis indicates the corresponding simulation is based
on the step size $h=2^{-i}$. The finest step size here is $2^{-8}$. The two sets of error
data are fitted with a linear function via linear regression respectively, where the
slope indicates the average order of convergence. It is noted that the
classical Euler scheme has a worse performance over $[2\tau,3\tau]$. The experimental orders of convergence from both methods are higher than the theoretical results shown in Theorem \ref{rate_of_conv_expl_Eul}).
\begin{figure}
\begin{subfigure}{0.475\textwidth}
    \centering
    \includegraphics[width=\textwidth]{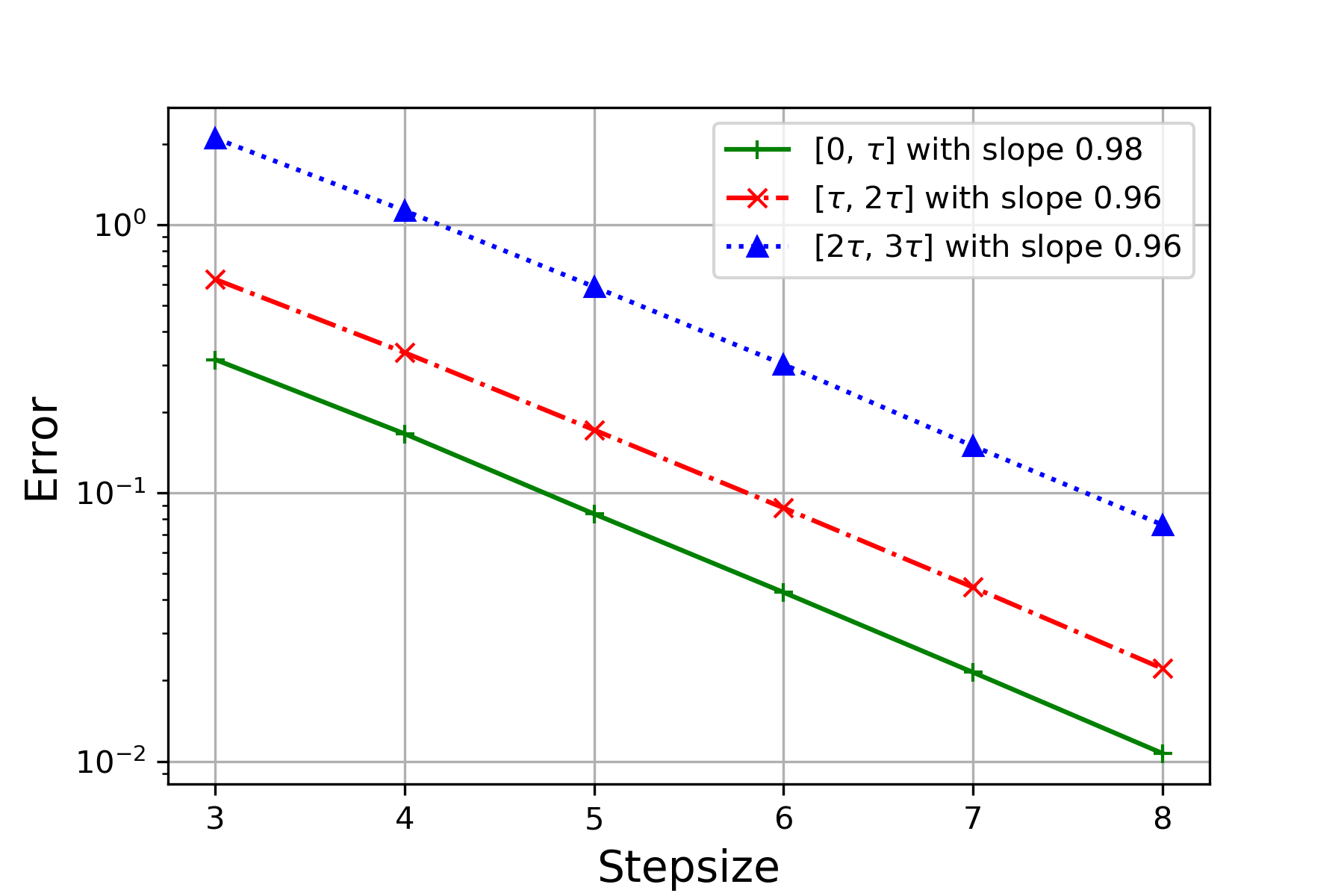}
    \caption{Randomized Euler method}
    \label{fig:comparison_BM_reuler}
\end{subfigure}
\begin{subfigure}{0.475\textwidth}
    \centering
    \includegraphics[width=\textwidth]{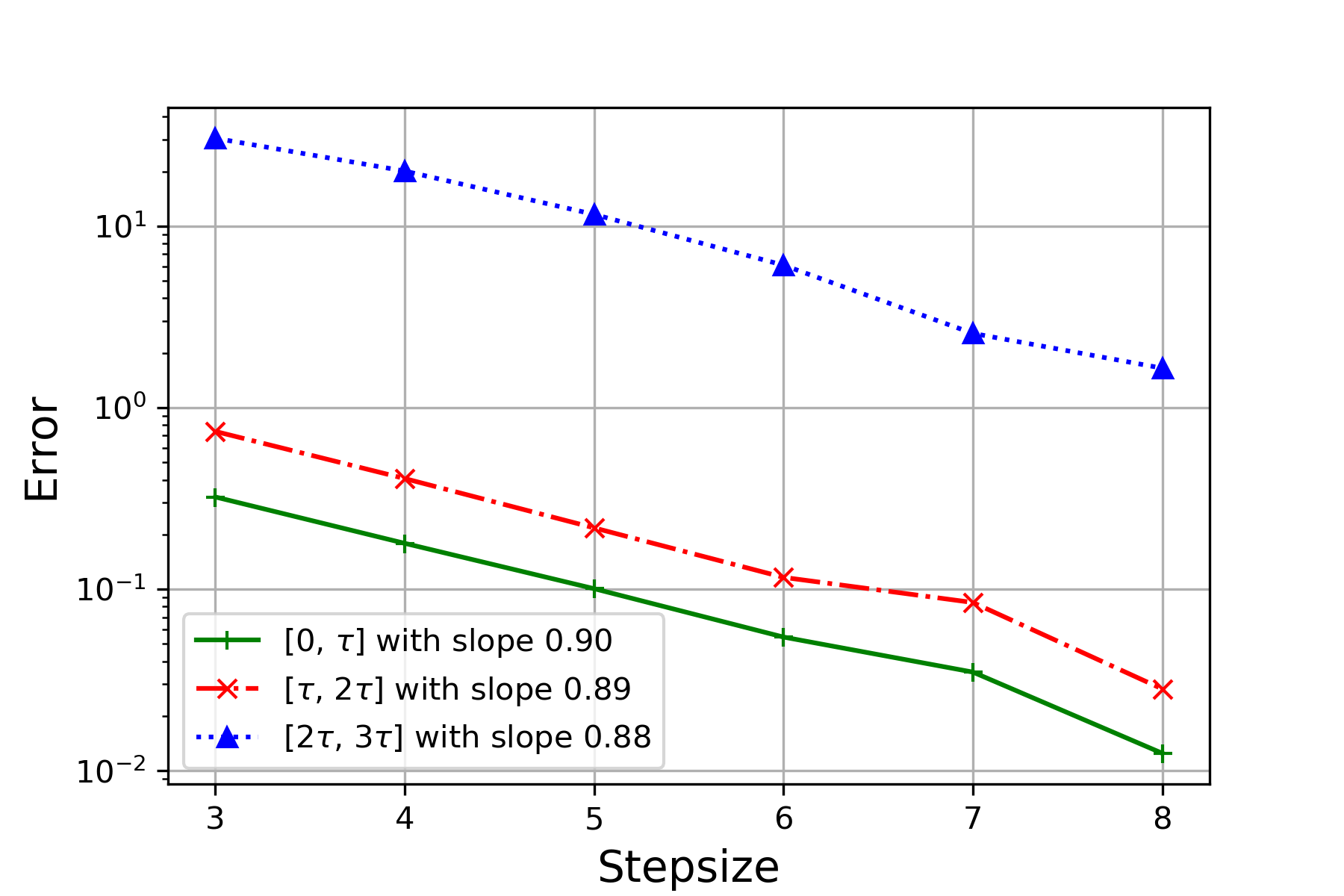}
    \caption{Euler method}
    \label{fig:comparison_BM_euler}
\end{subfigure}
\caption{Error comparison between randomized Euler method and Euler method for Example \ref{ex:BM} over $[0,\tau]$, $[\tau,2\tau]$ and $[2\tau, 3\tau]$.}
\label{fig:comparison_BM}
\end{figure}
\end{exm}
%%%%%%%%%%%%%%%%%
\section{Conclusions}

We investigated existence, uniqueness and numerical approximation of solutions of Carath\'eodory DDEs. In particular, we showed upper bound on the $L^p(\Omega)$-error for the randomized Euler scheme under global Lipschitz/H\"older  condition \ref{ass:A3'}. We conjecture however that the established upper error bound also holds under weaker local Lipschitz assumption. We plan to address this topic in our future work.
%%%%%%%%%%%%%%%%%
\section{Appendix}
We  use the following result  concerning properties of solutions of Carath\'eodory ODEs. It follows from
\cite[Theorem 2.12, pag. 252]{andresgorniewicz1}.  (Compare also with \cite[Proposition 4.2]{RKYW2017}.) 
\begin{lem}
	\label{lem_ode_1}
	Let us consider the following ODE
	\begin{equation}
	\label{ODE_1_Peano}
	    z'(t)=g(t,z(t)), \quad t\in [a,b], \quad z(a)=\xi,
	\end{equation}
	where $-\infty<a<b<+\infty$, $\xi\in\R^d$ and $g:[a,b]\times\R^d\to\R^d$ satisfies the following conditions
	\begin{enumerate}[label=\textbf{(G\arabic*)},ref=(G\arabic*)]
		\item\label{ass:G1}
		for all $t\in [a,b]$ the function $g(t,\cdot):\R^d\to\R^d$ is continuous,
		\item\label{ass:G2} 
		for all $y\in\R^d$ the function $g(\cdot,y):[a,b]\to\R^d$ is Borel measurable,
		\item\label{ass:G3} 
		there exists $K:[a,b]\to [0,+\infty)$ such that $K\in L^1([a,b])$  and for all $(t,y)\in [a,b]\times\R^d$
		\begin{displaymath}
		\|g(t,y)\|\leq K(t)(1+\|y\|),
		\end{displaymath}
		\item\label{ass:G4} 
		for every compact set $U\subset\R^d$ there exists $L_U:[a,b]\to\R^d$ such that $L_U\in L^1([a,b])$ and for all $t\in [a,b]$, $x,y\in U$
		\begin{equation}
		    \|g(t,x)-g(t,y)\|\leq L_U(t) \|x-y\|.
		\end{equation}
	\end{enumerate}	
	Then \eqref{ODE_1_Peano} has a unique absolutely continuous solution $z:[a,b]\to\R^d$ such that
	\begin{equation}
	\label{est_sol_z}
	\sup\limits_{t\in[a,b]}\|z(t)\|\leq (\|\xi\|+\|K\|_{L^1([a,b])})e^{\|K\|_{L^1([a,b])}}.
	\end{equation}
	Moreover, if $K\in L^p([a,b])$ for some $p\in (1,+\infty]$,  then for all $t,s\in [a,b]$
	\begin{equation}
	\label{lip_sol_z}
	\|z(t)-z(s)\|\leq \bar K |t-s|^{1-\frac{1}{p}},
	\end{equation}
	where $\bar K=\|K\|_{L^p([a,b])}\cdot\Bigl(1+(\|\xi\|+\|K\|_{L^1([a,b])})e^{\|K\|_{L^1([a,b])}}\Bigr)$.
\end{lem}
\begin{proof}
    By \ref{ass:G1}, \ref{ass:G2}, \ref{ass:G3} and by applying Theorem (2.12) from \cite[pag. 252]{andresgorniewicz1} to the set-valued mapping $F(t,x)=\{g(t,x)\}$, we get that \eqref{ODE_1_Peano} has at least one absolutely continuous solution. Moreover, any solution $z$ of \eqref{ODE_1_Peano} satisfies for all $t\in [a,b]$
    \begin{equation}
        \|z(t)\|\leq \|\xi\|+\int\limits_a^t\|g(s,z(s))\|\,\de s\leq \|\xi\|+\|K\|_{L^1([a,b])}+\int\limits_a^t K(s) \|z(s)\|\,\de s
    \end{equation}
    and by applying Gronwall's lemma (see \cite[pag. 22]{PLBR}) we get the estimate \eqref{est_sol_z}. 
    
    Let us consider the ball
    \begin{equation}
        B_0=\{y\in\mathbb\R^d \ | \ \|y\|\leq (\|\xi\|+\|K\|_{L^1([a,b])})e^{\|K\|_{L^1([a,b])}}\},
    \end{equation}
    which is a compact subset of $\R^d$. Let $z$ and $\tilde z$ are two solutions to \eqref{ODE_1_Peano}. From the consideration above we know that $z(t),\tilde z(t)\in B_0$ for all $t\in [a,b]$. Hence, by \ref{ass:G4} applied to $U:=B_0$ we get that there exists non-negative $L_U\in L^1([a,b])$ such that for all $t\in [a,b]$ 
    \begin{equation}
        \|z(t)-\tilde z(t)\|\leq \int\limits_a^t\|g(s,z(s))-g(s,\tilde z(s))\|\,\de s\leq\int\limits_a^t L_U(s) \cdot\|z(s)-\tilde z(s)\|\,\de s.
    \end{equation}
This and Gronwall's lemma (see \cite[pag. 22]{PLBR}) imply that $z(t)=\tilde z(t)$ for all $t\in [a,b]$.

Finally, for all $s,t\in [a,b]$ we get by the H\"older inequality that
\begin{equation}
    \|z(t)-z(s)\|\leq (1+\sup\limits_{a\leq t\leq b}\|z(t)\|)\int\limits_{\min\{t,s\}}^{\max\{t,s\}}K(u)\,\de u
\end{equation}
\begin{equation}
\label{est_hold_z}
    \leq (1+\sup\limits_{a\leq t\leq b}\|z(t)\|)\cdot\|K\|_{L^p([a,b])}\cdot|t-s|^{1-\frac{1}{p}}.
\end{equation}
Hence, by \eqref{est_hold_z} and \eqref{est_sol_z} we get \eqref{lip_sol_z}.
\end{proof}
%%%%%%%%%%%%%%%%%%
\begin{rem} For the readers familiar with stochastic differential equations we stress that the existence and uniqueness of solution to \eqref{ODE_1_Peano} under the assumptions \ref{ass:G1}-\ref{ass:G4} can also be derived from \cite[Proposition 3.28, pag. 187]{PARRAS}.
\end{rem}
%%%%%%%%%%%%%%%%%%%%%%%%
\section*{Acknowledgments}

F.V. Difonzo has been supported by \textit{REFIN} Project, grant number 812E4967; he also gratefully thanks INdAM-GNCS group for partial support.

P.~Przyby{\l}owicz is supported  by  the  National  Science  Centre, Poland, under project 2017/25/B/ST1/00945.
%%%%%%%%%%%%%%%%%%%%%%%%
\bibliographystyle{plain}
\bibliography{biblioDDDE.bib}

\begin{thebibliography}{10}

\bibitem{andresgorniewicz1}
J.~Andres and L.~G\'orniewicz.
\newblock {\em Topological Fixed Point Principles for Boundary Value Problems,
  vol. I}.
\newblock Springer Science+Business Media Dordrecht, 2003.

\bibitem{bellen1}
A.~Bellen and M.~Zennaro.
\newblock {\em Numerical methods for delay differential equations}.
\newblock Oxford, New York, 2003.

\bibitem{BGMP2021}
T.~Bochacik, M.~Go\'cwin, P.~M. Morkisz, and P.~Przyby{\l}owicz.
\newblock Randomized {R}unge-{K}utta method--{S}tability and convergence under
  inexact information.
\newblock {\em J. Complex.}, 65:101554, 2021.

\bibitem{bochacik2}
T.~Bochacik and P.~Przyby\l owicz.
\newblock On the randomized {E}uler schemes for {ODEs} under inexact
  information.
\newblock {\em to appear in Numerical Algorithms}, 2022.

\bibitem{CZPMPP}
N.~Czy\.z{}ewska, P.~M. Morkisz, and P.~Przyby{\l}owicz.
\newblock Approximation of solutions of {DDEs} under nonstandard assumptions
  via {Euler} scheme.
\newblock {\em https://arxiv.org/abs/2106.03731}, 2022.

\bibitem{daun1}
T.~Daun.
\newblock On the randomized solution of initial value problems.
\newblock {\em J. Complex.}, 27:300--311, 2011.

\bibitem{hale1977theory}
J.~K. Hale.
\newblock {\em Theory of Functional Differential Equations}.
\newblock Applied Mathematical Sciences. Springer New York, 1977.

\bibitem{Hale1993IntroductionTF}
J.~K. Hale and S.~M.~V. Lunel.
\newblock {\em Introduction to Functional Differential Equations}.
\newblock Springer-Verlag, New York, 1993.

\bibitem{hein_milla1}
S.~Heinrich and B.~Milla.
\newblock The randomized complexity of initial value problems.
\newblock {\em J. Complex.}, 24:77--88, 2008.

\bibitem{jen_neuen1}
A.~Jentzen and A.~Neuenkirch.
\newblock A random {E}uler scheme for {C}arath\'eodory differential equations.
\newblock {\em J. Comp. Appl. Math.}, 224:346--359, 2009.

\bibitem{BK2006}
B.~Kacewicz.
\newblock Almost optimal solution of initial-value problems by randomized and
  quantum algorithms.
\newblock {\em J. Complex.}, 22:676--690, 2006.

\bibitem{kainhofer}
R.~Kainhofer.
\newblock {QMC} methods for the solution of delay differential equations.
\newblock {\em Journal of Computational and Applied Mathematics},
  155(2):239--252, jun 2003.

\bibitem{KloNeu2007}
P.~E. Kloeden and A.~Neuenkirch.
\newblock The pathwise convergence of approximation schemes for stochastic
  differential equations.
\newblock {\em LMS J. Comput. Math.}, 10:235--253, 2007.

\bibitem{RKYW2017}
R.~Kruse and Y.~Wu.
\newblock Error analysis of randomized {R}unge--{K}utta methods for
  differential equations with time-irregular coefficients.
\newblock {\em Comput. Methods Appl. Math.}, 17:479--498, 2017.

\bibitem{kruse2017randomized}
Raphael Kruse and Yue Wu.
\newblock A randomized milstein method for stochastic differential equations
  with non-differentiable drift coefficients.
\newblock {\em arXiv preprint arXiv:1706.09964}, 2017.

\bibitem{novak1}
E.~Novak.
\newblock {\em Deterministic and Stochastic Error Bounds in Numerical
  Analysis}.
\newblock Lecture Notes in Mathematics, vol. 1349, New York, Springer-Verlag,
  1988.

\bibitem{PARRAS}
E.~Pardoux and A.~Rascanu.
\newblock {\em Stochastic Differential Equations, Backward SDEs, Partial
  Differential Equations}.
\newblock Stochastic Modelling and Applied Probability. Springer International
  Publishing Switzerland, 2014.

\bibitem{PLBR}
E.~Platen and N.~Bruti-Liberati.
\newblock {\em Numerical Solution of Stochastic Differential Equations with
  Jumps in Finance}.
\newblock Stochastic Modelling and Applied Probability. Springer--Verlag Berlin
  Heidelberg, 2010.

\end{thebibliography}
\end{document}